\documentclass[a4paper,11pt,oneside,openright,reqno]{amsart}
\usepackage[english]{babel}

\usepackage{amsthm}
\usepackage{color}
\usepackage{mathrsfs}
\usepackage{amsmath}
\usepackage{mathtools}
\usepackage{amsfonts}
\usepackage{amssymb}
\usepackage{bm}
\usepackage{physics}
\usepackage{enumitem}
\usepackage{a4wide}
\usepackage[margin=1in]{geometry}

\usepackage{esint}
\usepackage{nicefrac}
\usepackage{yfonts}
\usepackage{dirtytalk}
\usepackage{xcolor}
\usepackage[colorlinks,citecolor=green,linkcolor=red]{hyperref}

\numberwithin{equation}{section}

\newtheorem{thm}{Theorem}[section]
\newtheorem*{thm*}{Theorem}
\newtheorem{lem}[thm]{Lemma}
\newtheorem{prop}[thm]{Proposition}
\newtheorem{cor}[thm]{Corollary}

\theoremstyle{definition}
\newtheorem{defn}[thm]{Definition}

\theoremstyle{remark}
\newtheorem{rem}[thm]{Remark}

\newcommand{\fr}{}

\def\XXint#1#2#3{{\setbox0=\hbox{$#1{#2#3}{\int}$}
		\vcenter{\hbox{$#2#3$}}\kern-.5\wd0}}

\newcommand{\RCD}{{\mathrm {RCD}}}

\newcommand{\XX}{{\mathsf{X}}}
\newcommand{\YY}{{\mathsf{Y}}}

\newcommand{\dist}{{\mathsf{d}}}
\newcommand{\mass}{{\mathsf{m}}}

\newcommand{\LIP}{{\mathrm {LIP}}}

\newcommand{\DIFF}{{\mathrm{D}}}

\newcommand{\RR}{\mathbb{R}}

\newcommand{\HH}{\mathcal{H}}
\newcommand{\LL}{\mathcal{L}}

\newcommand{\DDelta}{{\mathbf{\Delta}}}
\newcommand{\Ric}{{\bf {Ric}}}
\newcommand{\heat}{{\mathrm {h}}}

\newcommand{\hess}{{\mathrm{Hess}}}

\newcommand{\defeq}{\vcentcolon=}
\newcommand{\eqdef}{=\vcentcolon}
\newcommand{\capa}{{\mathrm {Cap}}}

\let\oldchi=\chi
\renewcommand{\chi}{\text{\raisebox{\depth}{\(\oldchi\)}}}
\let\phi\varphi
\let\epsilon\varepsilon

\title[Perelman's entropy  and heat kernel bounds on RCD spaces]{Perelman's entropy  and heat kernel bounds \\on RCD spaces}
\author{Camillo Brena}
\address{Institute for Advanced Study. Einstein Drive 1, 08540 Princeton, New Jersey}
\email{cbrena@ias.edu}

\begin{document}
	\raggedbottom

	\begin{abstract}
    We study Perelman's $\mathcal{W}$-entropy functional on finite-dimensional $\mathrm{RCD}$ spaces, a synthetic generalization of spaces with Bakry--\'{E}mery Ricci curvature bounded from below. We rigorously justify the formula for the time derivative of the $\mathcal{W}$-entropy and derive its monotonicity and rigidity properties. Additionally, we establish bounds for solutions of the heat equation, which are of independent interest. 
	\end{abstract}
		
\maketitle
	\tableofcontents
\section{Introduction}
The focus of this  paper is Perelman's $\mathcal{W}$ entropy on non-smooth $\RCD$ spaces. Perelman's $\mathcal{W}$ entropy functional  was introduced  on Ricci flow in \cite{PerelmanEntropy} and played an important role in the resolution of the Poincaré conjecture. Two years later, in \cite{NiEntropy,NiEntropyAdd},  Ni noticed the similarity between the derivation of the entropy formula in \cite{PerelmanEntropy} and the gradient estimate for the heat equation established by Li and Yau in \cite{LiYauActa}. This led him to the intuition that a similar entropy formula exists for the heat equation. As a result, Li adapted  Perelman's entropy
  to the static case (i.e.\ for solutions of the heat equation).  
 The latter is the functional we focus on, and its definition will be recalled shortly.  $\RCD(K,N)$ spaces are infinitesimally Hilbertian metric measure spaces that, in a synthetic sense, satisfy a lower bound (by $K$) on the Ricci curvature and an upper bound (by $N$) on the dimension, see \cite{Villani2017,AmbICM, gigli2023giorgi} and references therein. Actually,  the bounds $K$ and $N$ should be taken into account simultaneously in the so-called Curvature-Dimension condition (\cite{Lott-Villani09} and  \cite{Sturm06I,Sturm06II}). For example, a smooth weighted Riemannian manifold $(M^n, g, e^{-V}d\mathrm{Vol})$ is $\RCD(K,N)$ if and only if
 \begin{equation}\label{vefrdsacas}
 	\mathrm{Ric}_N\defeq\mathrm{Ric}+\hess(V)-\frac{dV\otimes dV}{N-n}\ge K,
 \end{equation}
 where the fraction has to be understood as $\infty$ if $n>N$, or if $n= N$ and $dV\ne 0$. 
 Among $\RCD$ spaces, a more regular subclass is the one of non-collapsed spaces, \cite{DPG17}. More precisely, non-collapsed  $\RCD$ spaces are those $\RCD(K,n)$ spaces whose reference measure is exactly $\HH^n$, and are characterized by the fact that the (distributional) Laplacian equals the trace of the Hessian, \cite{DPG17,H19,brena2021weakly}. In words, this corresponds to the unweighted case. The class of  $\RCD$ spaces encompasses  Ricci-limit spaces, i.e.\  those spaces arising as  Gromov--Hausdorff limits of  Riemannian manifolds with Ricci curvature uniformly bounded from below, \cite{Cheeger-Colding97I,Cheeger-Colding97II,Cheeger-Colding97III}.

\bigskip

Now we recall the definition of the $\mathcal{W}$ entropy functional and we
highlight two key properties that it satisfies. For an $\RCD(K,N)$ space,  define
\begin{equation}\label{casdcvsavd1}	
	\mathcal{W}_t \defeq \int (t|\nabla f|^2+f-N)\rho_t\dd y,
\end{equation}
where the point $x$ is fixed, $\rho_t(x,y)$ denotes the heat kernel, and 
\begin{equation}\label{casdcvsavd}
	f=f_t(x,y) \defeq -\log \rho_t -\frac{N}{2}\log t-\frac{N}{2}\log(4\pi).
\end{equation} 
For the following statement, which serves as an example of the results of this note, we restrict  ourselves to the case $K\ge 0$.
\begin{thm*} Let us consider an $\RCD(0,N)$ space. Then, the following hold.
	\begin{description}
	\item[\quad Monotonicity] $t\mapsto\mathcal{W}_t$ is non-increasing, more precisely
	\begin{equation}\label{vefdcvasdc}
		\mathcal{W}_{t_1}\ge \mathcal{W}_{t_2}\quad\text{for every $0<t_1<t_2<\infty$}.
	\end{equation}
	\item[\quad Rigidity] If $t\mapsto\mathcal{W}_t$ is not \emph{strictly} decreasing, i.e.\ equality is obtained in \eqref{vefdcvasdc} for some choice $0<t_1<t_2<\infty$, then the space is a cone. 
	\end{description}
\end{thm*}

It is needless to point out to the reader the importance that monotone and rigid quantities play in geometric analysis. However, we discuss one example, as it is the main motivation of our work. The example under consideration is the main result of \cite{CJNrect}, where the singular set of non-collapsed Ricci limit spaces is analyzed, proving its rectifiability along with sharp Minkowski estimates. 
In \cite[Section 6]{CJNrect}, an important part is played by the sharp (quantitative) cone-splitting  principle, in which is proved the existence of a splitting map satisfying  quantitative estimates depending upon the pinching of the $\mathcal{W}$ entropy. The proof of this principle follows from the fact that, for a smooth Riemannian manifold $(M^n,g)$, 
		\begin{equation}\label{vfedcdasc}
	\begin{split}
		\partial_t \mathcal{W}_{t}(x)&= -2t \int_\XX\bigg(\Big| \hess f-\frac{1}{2t}g\Big|^2+\mathrm{Ric}(\nabla f,\nabla f)\bigg)\rho_t\dd y.
	\end{split}
\end{equation}
Hence, \eqref{vfedcdasc} makes apparent that small pinching of the entropy (which is obtained when the manifold is close to be conical, cf.\ the \textbf{Rigidity} statement above), implies the existence of a function (actually, the logarithm of the heat kernel) whose Hessian is close to be a multiple of the metric, in a quantitative sense. This is the starting observation for the proofs of the results of  \cite[Section 6]{CJNrect}.

In order to analyze singular sets of Ricci-limit space, it is enough to have \eqref{vfedcdasc} for smooth Riemannian manifolds, and this is by now a well-known formula, \cite{NiEntropy,NiEntropyAdd}. However, if one tries to extend the argument of \cite{CJNrect} to non-collapsed $\RCD$ spaces, a smooth approximating sequence is not available, hence, it seems necessary to obtain the analogue of \eqref{vfedcdasc} directly in the non-smooth realm. This is indeed the result (see Corollary \ref{bvrfdca}) that motivated the present paper, with the future goal of generalizing the results of \cite{CJNrect}  to non-collapsed $\RCD$ spaces in \cite{cbprep}. To be more precise, \cite{CJNrect} uses a modified entropy, where the integrand is multiplied by a compactly supported cut-off function. From a technical perspective, the derivative of this modified functional can also be more easily computed on non-smooth spaces.
 
It is worth mentioning that monotonicity and rigidity of the entropy on $\RCD(0,N)$ spaces are already known from \cite{KuwadaLi}, see also \cite{JiangZhang16} for the compact case. However, no proof of the analogue of \eqref{vfedcdasc} is obtained, and, even worse, it seems that there is no clear way to adapt the proof of \cite{KuwadaLi} to obtain \eqref{vfedcdasc}. In \cite{KuwadaLi},  the authors mention 
\begin{itemize}
	\item the lack of smoothness of $\RCD$ spaces, which makes the  second order calculus on these spaces (which is needed to state \eqref{vfedcdasc})  a remarkable achievement,
	\item the fact $\RCD$ spaces do not have ``bounded geometry'', 
\end{itemize} as the main difficulties towards obtaining the analogue of \eqref{vfedcdasc} in our setting. In this note, we overcome these issues, respectively, by
\begin{itemize}
	\item  exploiting heavily   the technical machinery that has been developed  so far on $\RCD$ spaces, 
	\item differentiating in time a modified version of the entropy (in which we multiply the integrand by a cut-off function) and then carefully taking the limit as the cut-off function converges to $1$.
\end{itemize} 
On the other hand, in \cite{KuwadaLi} the authors overcame these issues  by bypassing the computations leading to \eqref{vfedcdasc}, and  exploiting instead the control on the evolution of the Wasserstein distance for measures evolving via heat flow (which can be used to characterize the $\RCD$ condition). To draw a parallel, \cite{KuwadaLi} uses Lagrangian arguments, whereas the present note uses only Eulerian arguments and hence succeeds in obtaining the formula describing the derivative in time of  Perelman's $\mathcal{W}$ entropy.

We refer to \cite{KuwadaLi} and the references therein for a list of related works where similar results have been obtained.

\subsection{Bounds for solutions of the heat flow equation}
In \cite[Theorem 4.14]{CJNrect}, the authors recalled some results concerning heat kernel bounds that are used to help with the computations for the entropy. Even though these results are not all strictly necessary, we begin the paper by recalling relevant estimates.

Our first statement concerns bounds for solutions of the heat equation, and the heat kernel estimates just mentioned will follow, see Theorem \ref{hkest}. While most of these estimates are already well-known, we single out these bounds in the following statement for the reader's convenience. 
To discuss further the three bounds, we recall that items $\it(1)$ and $\it(2)$
are gradient bounds in the spirit of \cite{LiYauActa}. In the smooth context, $\it(1)$ is due to \cite{Hamilton93}. For $\RCD$ spaces, such result has been proved  in \cite{JiangZhang16} and we refer to that paper for references. Item $\it(2)$ is the Baudoin--Garofalo inequality, \cite{BaudoinGarofalo}, and  is proved for $\RCD$ spaces in   \cite{Jiang15}, or  \cite{GarofaloMondino14} in the  case of finite mass. Notice that for $K=0$ it reads as the well-known Li--Yau gradient bound. Item $\it (3)$ is obtained by Hamilton \cite{Hamilton93} for  Riemannian manifolds and is generalized in this paper for $\RCD$ spaces. The improvement is in two directions: first, we deal with possibly collapsed spaces  (the weighted case), second, we are able to treat non-smooth spaces.

\begin{thm}\label{grfeac}
	Let $(\XX,\dist,\mass)$ be an $\RCD(K,N)$ space. Let $u\in L^q(\XX)\cap L^\infty(\XX)$ for some $q\in [1,\infty)$ with $u\ge 0$, not vanishing identically. Then the following hold.
	\begin{enumerate}
		\item  For every $t>0$,
		\begin{equation}\notag
			t|\nabla \log h_t u|^2\le (1+2K^- t)\log\Big(\frac{\|u\|_{L^\infty}}{h_tu}\Big)\quad\quad\text{for $\mass$-a.e.\ $x$}.
		\end{equation}
		\item For every $t>0$,
		\begin{equation}\notag
			|\nabla \log h_t u |^2\le e^{2K^-t/3}\frac{\Delta h_t u}{h_t u}+\frac{NK^-}{3}\frac{e^{4K^-t/3}}{e^{2K^-t/3}-1}\quad\text{for  $\mass$-a.e.\ $x$}.
		\end{equation}
		\item For every $t>0$,
		\begin{equation}\label{veafdcs}
			t \Delta h_t u\le e^{K^-t} h_t u\bigg(N+4\log\Big(\frac{\|u\|_{L^\infty}}{h_tu}\Big)\bigg)\quad\text{for  $\mass$-a.e.\ $x$}.
		\end{equation}
	\end{enumerate}
\end{thm}
	In the statement above and in the rest of the paper, we denote $K^-\defeq -K\vee 0$.

\bigskip
Next, we specialize the bounds of Theorem \ref{grfeac} to the heat kernel on non-collapsed $\RCD$ spaces.
 First, we recall the definition of $f$, \eqref{casdcvsavd}, which will be used also in the definition of the entropy functional. In this note, $\rho_t(x,y)$ denotes the heat kernel.
	\begin{defn}\label{vefdca}
	Let  $(\XX,\dist,\mass)$ be an $\RCD(K,N)$ space. Define
	\begin{equation}\notag
		f_t(x,y)\defeq -\log \rho_t(x,y) -\frac{N}{2}\log t-\frac{N}{2}\log(4\pi).
	\end{equation}
\end{defn}
Notice that we trivially have 	\begin{equation}\notag
	\rho_t(x,y)=(4\pi t)^{-N/2}e^{-f_t(x,y)}.
\end{equation} 
Clearly, $f_t(x,y)$ also depends on $N$, however, we will not write this dependence explicitly.

 It is easy to realize that the bounds of the following Theorem \ref{hkest} are rather sharp.
 	\begin{thm}[Heat kernel estimates]\label{hkest}
	Let $(\XX,\dist,\HH^n)$ be a non-collapsed $\RCD(-(n-1)\delta^2,n)$ space, for $\delta>0$,  and let $p\in\XX$. Assume that $v\in (0,1)$ is such that $\HH^n(B_r(p))\ge v r^n$ for any $r\in(0,\delta^{-1})$. Then, for any $t\in (0,10\delta^{-2})$ and $x\in B_{10\delta^{-1}}(p)$, the following holds. For every $\epsilon\in (0,1)$, there exists a constant $C=C(n,v,\epsilon)$ such that
	\begin{enumerate}
		\item $-C+\frac{\dist^2(x,y)}{(4+\epsilon)t}\le f_t(x,y)\le C+\frac{\dist^2(x,y)}{(4-\epsilon)t}$ for every $y\in B_{10\delta^{-1}}(p)$,
		\item $t|\nabla_y f_t|^2(x,y)\le C +(1+C\delta^2t)\frac{\dist^2(x,y)}{(4-\epsilon)t}$ for $\mass$-a.e.\ $y\in B_{10\delta^{-1}}(p)$,
		\item  $-C\Big(1+\frac{\dist^2(x,y)}{(4-\epsilon)t}\Big)\le t\Delta_y f_t(x,y)\le C +(1+C\delta^2t)\frac{\dist^2(x,y)}{(4-\epsilon)t}$ for $\mass$-a.e.\ $y\in B_{10\delta^{-1}}(p)$.
	\end{enumerate}
	In particular,  for every $t\in (0,10\delta^{-2})$ and $x\in B_{10\delta^{-1}}(p)$,
	\begin{equation}\label{vaaweca}
		|f_t(x,y)|+t|\nabla_y f_t|^2(x,y)+t| \Delta_y f_t(x,y)|\le C(n,v,\epsilon)\Big(1+\frac{\dist^2(x,y)}{t}\Big)\quad\text{ for $\mass$-a.e.\ $y\in B_{10\delta^{-1}}(p)$}.
	\end{equation}
\end{thm}

\begin{rem}\label{vfdvfdcd}
	For Theorem \ref{hkest}, some effort is made, especially for items $\it (2)$ and $\it (3)$, to obtain bounds that do not involve exponential terms. We remark that,  exploiting naively the Gaussian bounds of Proposition \ref{gaussbound} below, we would get the following weaker bounds. Let $(\XX,\dist,\mass)$ be an $\RCD(K,N)$ space.
	Then, for every $t\in [a,b]\subseteq(0,\infty)$,
	\begin{equation}\label{vaawecaweak}
		\begin{split}
			|f_t(x,y)|\le C(K,N,\epsilon,a,b)\Big(1+|\log \mass(B_{\sqrt{t}}(x))\big|+ \frac{\dist^2(x,y)}{t}\Big)\quad\text{for every $x,y\in\XX$},\\
			t|\nabla_y f_t|^2(x,y)+t| \Delta_y f_t(x,y)|\le C(K,N,\epsilon, b) \exp(\epsilon\frac{\dist^2(x,y)}{t})\quad\text{for $(\mass\otimes\mass)$-a.e.\ $x,y\in\XX$}.
		\end{split}
	\end{equation}
	\fr
\end{rem}

\subsection{Perelman's entropy}
Now we recall the definition of Perelman's $\mathcal{W}$ entropy, \eqref{casdcvsavd1}.
In view of this, recall Definition \ref{vefdca}. Actually, our working definition of $\mathcal{W}$ entropy is going to be slightly different from \eqref{casdcvsavd1}  (see \cite{CJNrect}), but will turn out to be equivalent to \eqref{casdcvsavd1}, see Remark \ref{efdscacsdsa}.
\begin{defn}\label{defnentropy}
	Let  $(\XX,\dist,\mass)$ be an $\RCD(K,N)$ space. Define, for every $t>0$ and $x\in\XX$, the $L^\infty_{\mathrm{loc}}$ function  of $y$
	\begin{equation}\notag
		W_t(x,y)\defeq 
		2t\Delta_y f_t(x,y)-t|\nabla_y f_t|^2(x,y)+f_t(x,y)-N.
	\end{equation}
	Define also, for $t>0$ and $x\in\XX$,
	\begin{equation}\notag
		\mathcal{W}_t(x)\defeq \int_\XX W_t(x,y) \rho_t(x,y)\dd\mass(y).
	\end{equation}
\end{defn}
\begin{rem}\label{efdscacsdsa}
    We are going to prove that 
    \begin{equation}\notag
        \mathcal{W}_t(x)=\int_\XX (t|\nabla_y f_t|^2(x,y)+f_t(x,y)-N)\rho_t(x,y)\dd\mass(y),
    \end{equation}
    as the latter expression  may be found more frequently in the literature.
\end{rem}

Now we state a sample theorem, which addresses \textbf{Monotonicity} and \textbf{Rigidity} of the entropy on $\RCD(0,N)$ spaces, as well as  the formula for the time derivative of the entropy. We recall once more that, besides the formula \eqref{diffformula}, the following result has been proved in  \cite{KuwadaLi} with a different technique. Moreover, \eqref{diffformula} has been proved in the smooth context, under additional assumptions (that we are not going to make), in \cite{LiMathann}. In the smooth context, \eqref{vfedcdasc} and \eqref{diffformula} correspond, respectively, to the unweighted and weighted case.
We refer to Section \ref{bvefdcsacsd} for additional related results. In particular,  Theorem \ref{vefadcas} contains the precise formula for the derivative in time of the modified Perelman's entropy (i.e.\ in  presence of a cut-off function), Corollary \ref{bvrfdca} is the rigorous limit of the formula Theorem \ref{bvefdcsacsd} (where the cut-off is constantly equal to $1$) and Corollary \ref{corgrfdsa} should be compared to the results needed in \cite[Section 4.6]{CJNrect}. 

In the following result, a key role is played by the (modified) Ricci tensor on $\RCD$ spaces, which is defined in \cite{Han14}, see also \cite{Gigli14}. We are going to add a short discussion on the topic at the end of Section \ref{prelsect}.
\begin{thm}\label{entrointro}
	Let  $(\XX,\dist,\mass)$ be an $\RCD(0,N)$ space of essential dimension $n$ and let $x\in\XX$. Then, the following hold.
	\begin{description}
		\item[\quad Monotonicity] The curve $t\mapsto\mathcal{W}_{t}(x)$ is locally absolutely continuous in $(0,\infty)$,  is differentiable at every $t>0$, and is non-increasing. More precisely,
		\begin{equation}\label{diffformula}
\begin{split}
								\partial_t \mathcal{W}_{t}(x)&=-2t\int\bigg(	\Ric_N(\nabla_y f_t,\nabla_y f_t)+\Big| \hess_y f_t-\frac{1}{2t}g\Big|^2\\&\qquad\qquad\qquad+\frac{1}{N-n}\Big(\big(\tr(\hess f_t)-\Delta_y f_t\big)+\frac{N-n}{2t}\Big)^2\bigg)(x,y)\rho_t(x,y)\dd\mass(y).
\end{split}
		\end{equation}

		\item[\quad Rigidity] Assume that for some $t>0$, $\partial_t \mathcal{W}_{t}(x)=0$.  Then $t\mapsto \mathcal{W}_{t}(x)$ is constant.  Moreover
		\begin{itemize}
			\item either $(\XX,\dist,\mass)$  is isomorphic to the cone built over an $\RCD(N-2,N-1)$ space,  and $x$ corresponds to the tip, or
			\item  $(\XX,\dist,\mass)$  is isomorphic  either to $(\RR,\dist_e,c|x|^{N-1}\LL^1)$ or to $([0,\infty),\dist_e,c|x|^{N-1}\LL^1)$, for some $c>0$, and $x$ corresponds to $0$.
		\end{itemize}
	\end{description}
\end{thm}
To be more precise, the \textbf{Rigidity} statement of \cite{KuwadaLi} is slightly more powerful, as it is proved for any solution of the heat equation in place of the heat kernel (recall  Definition \ref{vefdca}). In both the proofs the main part of the argument is to obtain equality in the Li--Yau gradient bound, use the short time asymptotic for the heat kernel to deduce equality in the Laplacian comparison inequality, and finally  conclude by \cite{DPG16}.

\section*{Acknowledgements}
Part of this work was carried out while the author was a PhD student at Scuola Normale Superiore, during a visit to E.\ Bruè and A.\ Pigati at Bocconi University. The author wishes to thank E.\ Bruè and A.\ Pigati for many discussions around this topic. The author is grateful to L.\ Gennaioli for comments about an earlier draft of this work and to N.\ Gigli and X.-D.\ Li for discussions about the entropy functional.

	This material is based upon work supported by the National Science Foundation under Grant No.\ DMS-1926686.
\section{Preliminaries}\label{prelsect}

Our investigation is set in $\RCD(K, N)$ metric measure spaces (for $K\in\RR$ and $N\ge 1$). These are infinitesimally Hilbertian spaces, \cite{Gigli12}, that satisfy  a lower Ricci curvature bound and an upper dimension bound  in synthetic sense according to \cite{Sturm06I,Sturm06II}, \cite{Lott-Villani09}.  We assume familiarity with this subject throughout,  see \cite{Villani2017,AmbICM, gigli2023giorgi} and references therein. For notions of calculus on this spaces (up to the second order), we refer to \cite{Gigli14} or \cite{GP19}. We are now going to recall some of the notions used most frequently in this paper.

\bigskip

The following lemma from  \cite{Mondino-Naber14} is going to be useful for us.
\begin{lem}[{Good cut-off functions,  \cite[Lemma 3.1]{Mondino-Naber14}}]\label{goodcutoff}
	Let $(\XX,\dist,\mass)$ be an $\RCD(K,N)$ space. Then, for every $R>0$, for every $r\in (0,R),$ and $x\in\XX$, there exists a function $\varphi\in \LIP_{\mathrm{bs}}(\XX)\cap D(\Delta)$  satisfying
	\begin{enumerate}
		\item  $0\le \varphi\le 1$  on $\XX$, $\varphi=1$ on $B_r(x)$, and $\mathrm{supp}(\varphi) \subseteq B_{2r}(x)$;
		\item $r^2|\Delta\varphi|+ r|\nabla\varphi|\le C(K^{-}R^2,N)$.
	\end{enumerate}
\end{lem}

\bigskip

On $\RCD(K,N)$ spaces there is an established  theory for solutions of the heat equation. In particular, there exists a  locally H\"{o}lder continuous
heat kernel $\rho_t(x,y)$, by \cite{Sturm96II,Sturm96III}. We collect in the next statement some properties of the heat kernel that will be needed throughout.
For $\it (1)$ and $(2)$, see \cite{Sturm96II,Sturm96III} and \cite{jiang2014heat}. For $\it (3)$, see also \cite{Sturm96II} after \cite{SCheat}, or \cite{Davies}.
\begin{prop}[Gaussian bounds]\label{gaussbound}
	Let $(\XX,\dist,\mass)$ be an $\RCD(K,N)$ space. 
	For every $\epsilon\in (0,1)$, there exists a constant $C=C(K,N,\epsilon)$  such that we have the following Gaussian bounds.
	\begin{enumerate}
		\item For every $x,y\in\XX$ and $t>0$, 
		\begin{equation}\notag
			\frac{1}{C\mass(B_{\sqrt{t}}(x))}\exp(-\frac{\dist^2(x,y)}{(4-\epsilon)t}-Ct)\le \rho_t(x,y)\le \frac{C}{\mass(B_{\sqrt{t}}(x))}\exp(-\frac{\dist^2(x,y)}{(4+\epsilon)t}+Ct).
		\end{equation}
		\item For every $x\in\XX$ and $t>0$,  
		\begin{equation}\notag
			|\nabla_y \rho_t|(x,y)\le \frac{C}{\sqrt{t}\mass(B_{\sqrt{t}}(x))}\exp(-\frac{\dist^2(x,y)}{(4+\epsilon)t}+Ct)\quad\text{for $\mass$-a.e.\ $y\in\XX$}.
		\end{equation}
		\item For every $y\in\XX$ and $t>0$,  
		\begin{equation}\notag
			|\partial_t^k \rho_t|(x,y)\le \frac{C(K,N,\epsilon,k)}{{t^k}\mass(B_{\sqrt{t}}(x))}\exp(-\frac{\dist^2(x,y)}{(4+\epsilon)t}+Ct)\quad\text{for $\mass$-a.e.\ $x\in\XX$}.
		\end{equation}
	\end{enumerate}
	
\end{prop}

	\bigskip
	
	We will need also the notion of measure valued (distributional) Laplacian,  \cite{Gigli14}, that we denote by $\DDelta$. In particular, for $f\in H^{1,2}$, we say that $f\in D(\DDelta)$ if there exists a (unique) finite measure $\DDelta f$ such that 
	\begin{equation}
		\int \nabla f\,\cdot\,\nabla g\dd\mass=-\int g\dd\DDelta f\quad\text{for every $g\in \LIP_{\mathrm{bs}}(\XX)$}.
	\end{equation}  
	We will often use the localized counterpart of the previous definition, which leads to the domain $D_{\mathrm{loc}}(\DDelta)$. 
	By \cite{Savare13}, by an adaptation of the arguments of \cite{Bakry83}, it is proved that for every $f\in \mathrm{TestF}(\XX)$, it holds that $|\nabla f|^2\in D(\DDelta)$. 
	Here, as in \cite{Gigli14}, we denote the space of test functions as
	\begin{equation}
		\mathrm{TestF}(\XX)\defeq \{f\in \LIP_{\mathrm{b}}(\XX)\cap D(\Delta): \Delta f\in H^{1,2}\}.
	\end{equation}

	We state and prove an elementary lemma that will turn out to be useful. For the notion of quasi-continuous representative, see \cite{debin2019quasicontinuous}.
	\begin{lem}\label{bgrfdva1}
		Let $f,g\in D(\DDelta)\cap L^\infty\cap H^{1,2}$. Then $fg\in D(\DDelta)\cap L^\infty\cap H^{1,2}$ with
		$$
		\DDelta (fg)= f\DDelta g+ g\DDelta f+2\nabla f\,\cdot\,\nabla g,
		$$
		where we implicitly took the quasi-continuous representatives of $f$ and $g$.
	\end{lem}
	\begin{proof}
		Let $\varphi\in \LIP_{\rm bs}(\XX)$,
		\begin{equation}\label{dfavdsc1}
			\begin{split}
				\int_\XX \nabla (fg)\,\cdot\,\nabla \varphi\dd\mass &= \int_\XX f\nabla g\,\cdot\,\nabla \varphi\dd\mass+  \int_\XX g\nabla f\,\cdot\,\nabla \varphi\dd\mass \\&=\int_\XX \nabla g\,\cdot\,\nabla (\varphi f)\dd\mass+\int_\XX \nabla f\,\cdot\,\nabla ( \varphi g)\dd\mass-2\int_\XX \varphi\nabla f\,\cdot\,\nabla g\dd\mass.
			\end{split}
		\end{equation}
		Now,
		\begin{align*}
			\int_\XX \nabla g\,\cdot\,\nabla (\varphi f)\dd\mass=\lim_{t\searrow 0}\int_\XX\nabla g\,\cdot\, \nabla (\varphi\heat_t f)\dd\mass=-\lim_{t\searrow 0}\int_\XX\varphi\heat_t f\dd\DDelta g,
		\end{align*}
		and similarly for the second term at the right hand side of \eqref{dfavdsc1}. Now we can use dominated convergence, with the theory in \cite{debin2019quasicontinuous} (see in particular \cite[Proposition 2.13]{debin2019quasicontinuous}) and the fact that $\DDelta g\ll \capa$ and $\DDelta f\ll\capa$.
	\end{proof}
	
	\bigskip
	Now we  recall  the definition of the Ricci tensor on $\RCD$ spaces, $\Ric$,  given in \cite{Gigli14}:
	\begin{equation}\notag
		\DDelta\frac{|\nabla f|^2}{2}-\nabla f\,\cdot\,\nabla \Delta f=|\hess f|^2+\Ric(\nabla f,\nabla f)\quad\text{for every $f\in\mathrm{TestF}(\XX)$}.
	\end{equation}
	In \cite{Gigli14}, it has been proved that on an $\RCD(K,N)$ space,
	\begin{equation}\notag
		\Ric(\nabla f,\nabla f)\ge K |\nabla f|^2\mass\quad\text{for every $f\in\mathrm{TestF}(\XX)$},
	\end{equation}
	as measures.  In the case $N<\infty$ (which is the only case we are considering in this note), \cite{Han14} defined the modified Ricci tensor, $\Ric_N$, as follows: 
	\begin{equation}\notag
		\DDelta\frac{|\nabla f|^2}{2}-\nabla f\,\cdot\,\nabla \Delta f=|\hess f|^2+\Ric_N(\nabla f,\nabla f)+\frac{(\tr(\hess f )-\Delta f)^2}{N-n}\quad\text{for every $f\in\mathrm{TestF}(\XX)$},
	\end{equation}
	where  in this paper we denote by $n\le N$  the essential dimension of the space, \cite{bru2018constancy}. In the case $N=n$, then  the space is non-collapsed by \cite{H19,brena2021weakly} and hence $\tr(\hess(f))=\Delta f$, in particular, the convention that the fraction in the equation above is $0$ is meaningful.
	 In \cite{Han14}, it has been proved that on an $\RCD(K,N)$ space,
	\begin{equation}\label{improvedbochner}
	\Ric_N(\nabla f,\nabla f)\ge K |\nabla f|^2\mass\quad\text{for every $f\in\mathrm{TestF}(\XX)$},
\end{equation}
as measures.

Finally, we denote by $g$ the tensor representing the metric of the space, that is characterized as follows:
\begin{equation}
	g(\nabla f_1,\nabla f_2)=\nabla f_1\,\cdot\,\nabla f_2\quad\mass\text{-a.e.}
\end{equation}
for every $f_1,f_2\in H^{1,2}(\XX)$. Its existence follows from \cite{Gigli14}, using the fact that the space is infinitesimally Hilbertian.

\section{Bounds for solutions of the heat flow equation}

Now we go towards the proof of Theorem  \ref{grfeac}. We really need to prove only the last bound, which is a consequence of Hamilton's Lemma contained in the next section.
\subsection{Hamilton's Lemma}
The following lemma is the adaptation of \cite[Lemma 4.1]{Hamilton93} to our context (see also \cite[Theorem 1.2]{NITAM} regarding  the maximum principle used in \textbf{Step 3} of its proof). In adapting the arguments we have just mentioned, we face two main difficulties:  that we are considering possibly \emph{collapsed} $\RCD$ spaces (which correspond to the weighted case), and the \emph{non-smoothness} of $\RCD$ spaces. The first issue is  dealt with the improved Bochner inequality of \cite{Han14} and  linear algebra arguments, whereas to deal with the second issue delicate approximation arguments and careful computations are needed.
\begin{lem}\label{hamlem}
	Let $(\XX,\dist,\mass)$ be an $\RCD(K,N)$ space. Let $u\in L^1(\XX)\cap L^\infty(\XX)$ with $u\ge 0$, not vanishing identically. Set as in \cite{Hamilton93}
	\begin{equation}\notag
		v_t\defeq \frac{e^{K^- t}-1}{K^- e^{K^- t}}\bigg(\Delta h_t u+\frac{|\nabla h_t u|^2}{h_t u}\bigg)- h_t u\bigg(N+4\log\Big(\frac{\|u\|_{L^\infty}}{h_tu}\Big)\bigg).
	\end{equation}
	Then,  for every $t>0$,
	\begin{equation}\notag
		v_t\le 0\quad\mass\text{-a.e.}
	\end{equation}
\end{lem}
\begin{proof}
	By an immediate approximation argument, we see that it is enough to work with  $u=h_\sigma w$, for some $w\in \LIP_{\mathrm{bs}}(\XX)\cap D(\Delta)$ non-negative with $\Delta w \in H^{1,2}\cap L^\infty$ and $\sigma>0$. Fix $T>0$.
	\medskip\\\textbf{Step 1.} We begin with a few technical estimates, which we will use throughout.
	First, 
	\begin{equation}\label{brrvfedac}
		\sup_{t\in [0,T]}\|\Delta h_t u\|_{H^{1,2}\cap L^\infty}+	\| h_t u\|_{L^2\cap L^\infty}+\|\nabla h_t u\|_{L^2\cap L^\infty}+ \|h_t u\log(h_t u)\|_{L^2\cap L^\infty}<\infty
	\end{equation}
	by the properties of the heat flow. 
	Now, notice that $v_t\in L^2_{\mathrm{loc}}$ is locally absolutely continuous as a function of $t\in [0,\infty]$, and differentiable  at every $t\ge 0$ in the sense that for every bounded set $B$, $t\mapsto v_t\chi_B \in L^2$ has these properties.
	Also, we have the estimates, for every $B\subseteq\XX$ bounded set,
	\begin{equation}\label{veacfd1}
		\sup_{t\in [0,T]} \|\partial_t v\|_{L^2(B)}+\|v_t\|_{H^{1,2}(B)\cap L^\infty(B)}<\infty.
	\end{equation}
	This is obtained by explicit differentiation, taking into account that $h_t u$ is uniformly bounded from below on $[0,T]\times B$. 
	
	Also, for every $t\ge 0$, $v_t\in D(\DDelta)$, where we are slightly abusing the notation, as $\DDelta v_t$ is a Radon measure that can be not finite.
		
		Next, we claim that for some $\bar x\in\XX$, for every $t\in [0,T]$, if $C$  is a constant that depends on $\XX, T, \sigma$ and $w$, and may vary line to line, 
		\begin{equation}\label{cveadfsc}
			\frac{|\nabla h_t u|(x)}{h_t u(x)}\le C e^{\frac{2}{\sigma}\dist^2(x,\bar x)}\quad\text{for }\mass\text{-a.e.\ $x$}.
		\end{equation}
		Indeed,  we let $\bar B=B_{\bar r}(\bar x)$ be a  ball  of minimal radius such that $\mathrm{supp}(w)\subseteq \bar B$.  
		Notice that $h_t u$ is uniformly bounded from below on $[0,T]\times 2\bar B$, and $|\nabla h_t u|$ is uniformly bounded from above, hence \eqref{cveadfsc} is satisfied on $2 \bar B$.
		Now, notice that
		\begin{equation}\notag
			\dist(x,\bar x)/2\le \dist (x,y)\le 2\dist(x,\bar x)\quad\text{for every $x\in\XX\setminus 2\bar B$ and  $y\in\bar B$}.
		\end{equation}
		Therefore, for any $t\ge 0$, using \cite[Corollary 4.3]{Savare13} with the upper Gaussian bound for the heat kernel,
		\begin{equation}\notag
			|\nabla h_t u|(x)\le C h_{t+\sigma} \chi_{\bar B}\le  \frac{C}{\mass(B_{\sqrt{t+\sigma}}(x))}\int_{\bar B} e^{-\frac{\dist^2(x,y)}{5(t+\sigma)}}\dd\mass(y)\le \frac{C}{\mass(B_{\sqrt{t+\sigma}}(x))} e^{-\frac{\dist^2(x,\bar x)}{20(t+\sigma)}},
		\end{equation} 
		and similarly, using the lower Gaussian bound for the heat kernel, for some $\bar B'\subseteq\bar B$ (depending on $w$),
		\begin{equation}\notag
			h_t u(x)\ge \frac{1}{C \mass(B_{\sqrt{t+\sigma}}(x))}\int_{\bar B'} e^{-\frac{\dist^2(x,y)}{3(t+\sigma)}}\dd\mass(y)\ge \frac{1}{C \mass(B_{\sqrt{t+\sigma}}(x))}e^{-4\frac{\dist^2(x,\bar x)^2}{3(t+\sigma)}}.
		\end{equation}
		The two above equation conclude the proof of \eqref{cveadfsc}.	
		\medskip\\\textbf{Step 2}. We prove that
		\begin{equation}\label{subsolution}
			\partial_t v_t\le \DDelta v_t\quad\text{on $\{v_t\ge 0\}$},
		\end{equation}
		in the sense that $\partial_t v_t- \DDelta v_t$ is bounded by above by an absolutely continuous measure that vanishes on $\{v_t\ge 0\}$.
		
		Set $\psi(t)\defeq \frac{e^{K^- t}-1}{K^- e^{K^-t}}$, notice that $\psi'+K^-\psi=1$.
		Fix $t\ge 0$. We start by noticing that by Lemma \ref{bgrfdva1} and a localization argument,
		\begin{align*}
			(	\DDelta-\partial_{t})\frac{|\nabla h_t u|^2}{h_t u}&=\frac{2}{h_t u}\bigg(\frac{\DDelta|\nabla h_t u|^2-2\nabla h_t u\,\cdot\,\nabla\Delta h_t u}{2}+\frac{|\nabla h_t u|^4}{h_t u^2}-\frac{\nabla h_t u\,\cdot\,\nabla |\nabla h_t u|^2}{h_t u}\bigg)\\&\ge\frac{2}{h_t u}\bigg( |\hess h_t u|^2+\frac{(\tr(\hess h_t u)-\Delta h_t u)^2}{N-n}+K|\nabla h_t u|^2+\frac{|\nabla h_t u|^4}{h_t u^2}\\
			&\quad\quad\quad\quad\quad\quad\quad\quad\quad\quad-\frac{2\hess h_t u(\nabla h_t u\otimes\nabla h_t u)}{h_t u}\bigg),
		\end{align*} 
		where the inequality follows from \eqref{improvedbochner}.
		
		Now we claim that 
		\begin{align*}
			&|\hess h_tu|^2+\frac{(\tr(\hess h_tu)-\Delta h_tu)^2}{N-n}+\frac{|\nabla h_tu|^4}{h_t u^2}-\frac{2\hess h_tu(\nabla h_t u\otimes\nabla h_tu)}{h_tu}\\&\quad\ge \frac{1}{N}\Big(\Delta h_tu-\frac{|\nabla h_tu|^2}{h_tu}\Big)^2\quad\text{$\mass$-a.e.}
		\end{align*}
		By arguing locally and using an orthonormal basis of $L^2(T\XX)$, we can prove the above by means of linear algebra. Indeed, using orthonormal local coordinates and fixing $\mass$-a.e.\ $x\in\XX$,  we can set set $A\defeq \hess h_tu (x)\in \RR^{n\times n}$, $\ell\defeq \frac{\nabla h_tu}{\sqrt {h_tu}}(x)\in \RR^n$, $\lambda\defeq\Delta h_tu(x)\in\RR$. It remains then to prove that
		\begin{equation}\notag
			|A|^2+\frac{(\tr(A)-\lambda)^2}{N-n}+|\ell|^4-2A\,\cdot\,(\ell\otimes \ell)\ge \frac{1}{N}(\lambda-|\ell|^2)^2,
		\end{equation}
		i.e.\ 
		\begin{equation}\notag
			N|A-\ell\otimes \ell|^2+\frac{N}{N-n}{(\tr(A)-\lambda)^2}\ge( \lambda-|\ell|^2)^2.
		\end{equation}
		We now use Cauchy--Schwarz to prove 
		\begin{equation}\notag
			N|A-\ell\otimes \ell|^2\ge \frac{N}{n}(\tr(A)-|\ell|^2)^2
		\end{equation}
		whence the claim by the inequality $(a+b)^2\le pa^2+q b^2$ whenever $a,b>0$ and $p^{-1}+q^{-1}=1$.
		
		We thus conclude that 
		\begin{equation}\notag
			(	\partial_{t}-\DDelta)\frac{|\nabla h_t u|^2}{h_t u}\le - \frac{2}{N h_t u}\Big(\Delta h_tu-\frac{|\nabla h_tu|^2}{h_tu}\Big)^2+ 2K^-\frac{|\nabla h_t u|^2}{h_t u}.
		\end{equation}
		Now we can follow \cite{Hamilton93}. By direct computation, using Lemma \ref{bgrfdva1} and a localization  argument, recalling the identity  $\psi'+K^-\psi=1$,
		\begin{equation}\notag
			(\partial_t-\DDelta )v\le -\frac{2\psi}{N h_t u}\Big(\Delta h_t u-\frac{|\nabla h_t u|^2}{h_t u}\Big)^2+\psi' \Big(\Delta h_t u-\frac{|\nabla h_t u|^2}{h_t u}\Big)-2\frac{|\nabla h_t u|^2}{h_t u},
		\end{equation}
		in the sense of measures. Notice that the right hand side is a measure given by a function, we claim that it is non-positive $\mass$-a.e.\ on $\{v_t\ge 0\}$ and this will conclude the proof of \eqref{subsolution}. This follows exactly as in \cite{Hamilton93}, we give the argument for completeness. We have to distinguish three cases, notice that $A_1\cup A_2\cup A_3 \supseteq \{v_t\ge 0\}$.
		\begin{enumerate}
			\item On $A_1\defeq\big\{\Delta h_t u\le \frac{|\nabla h_t u|^2}{h_t u}\big\}$. Just notice that $\psi'\ge 0$.
			\item On $A_2\defeq \big\{\frac{|\nabla h_t u|^2}{h_t u}\le \Delta h_t u\le3 \frac{|\nabla h_t u|^2}{h_t u} \big\}$. This follows as $\psi'\le 1$.
			\item On $A_3\defeq  \big\{3 \frac{|\nabla h_t u|^2}{h_t u}\le \Delta h_t u \big\}\cap \{v_t\ge 0\}$.  First, 
			\begin{equation}\notag
				N h_t u\le \psi \Big(\Delta h_t u+\frac{|\nabla h_t u|^2}{h_t u}\Big)\le2 \psi \Big( \Delta h_t u-\frac{|\nabla h_t u|^2}{h_t u}\Big)\quad\text{{$\mass$-a.e.\ on $A_3\cap \{v_t\ge 0\}$, }}
			\end{equation}
			where the first inequality follows as $v_t \ge 0$, and the second one as $\Delta h_t u\ge 3 \frac{|\nabla h_t u|^2}{h_t u}$, and so the conclusion follows as $\psi'\le 1$.
		\end{enumerate}
		\medskip\textbf{Step 3}. We conclude by using the maximum principle adapted from \cite{NITAM}. Take $T_1,T_2\in [0,T)$, with $T_1>T_2$, we will consider $t\in [T_2,T_1]$. Let $\varphi\in\LIP_{\rm bs}(\XX)$. Define also $g_t(x)\defeq -\frac{\dist^2(x,\bar x)}{4(2T_1-T_2-t)}$, notice that 
		\begin{equation}\label{bvfeacds}
			|\nabla {g_t} |^2+\partial_t {g_t} = 0\quad\mass\text{-a.e.}
		\end{equation}
		and that $e^{g_t} \in \LIP_{\mathrm {b}}(\XX)$ with
		\begin{equation}\notag
			\sup_{t\in [0,T_1]} \||\nabla e^{g_t} |\|_{L^\infty}<\infty.
		\end{equation}
		
		Take any $t\ge 0$, and write $\mu\defeq(\DDelta-\partial_t) v_t=f_t\mass+\mu^s_t$, where $\mu^s_t\perp\mass$. By \textbf{Step 2}, $f_t\ge 0\ \mass$-a.e.\ on $\{v_t\ge 0\}$, and we see that also $\mu_t^s\ge 0$, in the sense that it is a non-negative measure (by the Bochner inequality, the already used \eqref{improvedbochner}  gives a   precise version).  
		Hence, we can write
		\begin{equation}\label{vfadvefad}
			0\le  \int_\XX \varphi^2 e^{g_t}  v^+_t\dd \DDelta v_t-   \int_\XX \varphi^2 e^{g_t}  v^+_t \partial_t v_t\dd\mass.
		\end{equation}
		Some comments are in order. For the first integral, we took the quasi-continuous representative for $v_t^+$, which is possible as $v_t^+\in H^{1,2}_{\mathrm{loc}}$.  Any other $\capa$-measurable representative of $v_t^+$ would anyway work in \eqref{vfadvefad}, but we need the following property
		\begin{equation}\notag
			\int_\XX \varphi^2 e^{g_t}  v^+_t\dd \DDelta v_t=-\int_\XX \nabla (\varphi^2 e^{g_t}  v^+_t)\,\cdot\,\nabla v_t\dd\mass,
		\end{equation}
		which can be proved by an immediate approximation argument based on the theory of \cite{debin2019quasicontinuous}, as in the proof of Lemma \ref{bgrfdva1}.
		
		Hence,
		\begin{equation}\label{vfadvefd}
			0\le 	-\int_\XX \nabla (\varphi^2 e^{g_t}  v^+_t)\,\cdot\,\nabla v_t\dd\mass-  \int_\XX \varphi^2 e^{g_t}  v^+_t \partial_t v_t\dd\mass.
		\end{equation}
		For what concerns the first integral in \eqref{vfadvefd},
		\begin{align*}
			-	\int_\XX \nabla (\varphi^2 e^{g_t}  v^+_t)\,\cdot\,\nabla v_t\dd\mass&=-	\int_\XX \varphi^2 e^{g_t}  |\nabla v^+_t|^2+\varphi^2 e^{g_t}  v_t^+\nabla {g_t} \,\cdot\,\nabla v_t+2\varphi e^{g_t}  v_t^+\nabla \varphi\,\cdot\,\nabla v_t\dd\mass
			\\&\le 2\int_\XX e^{g_t} (v_t^+)^2|\nabla \varphi|^2\dd\mass+\frac{1}{2} \int_\XX\varphi^2 e^{g_t}  (v_t^+)^2|\nabla g_t|^2\dd\mass.
		\end{align*}
		Now we manipulate the second integral. Let $\chi(s)$ be any $C^1$ function that satisfies $\chi(s)=0$ for $s\le 0$, $\chi(s)=s$ for $s\ge 1$, let then, for $\epsilon\in(0,1)$, $\chi_\epsilon(s)\defeq \epsilon\chi(t/\epsilon)$. Notice that $(\chi_\epsilon)_{\epsilon\in (0,1)}$ are equi-Lipschitz. By direct computation, we see that $t\mapsto\chi_\epsilon\circ v_t\in L^2_{\mathrm{loc}}$ and $t\mapsto(\chi_\epsilon\circ v_t)^2\in L^2_{\mathrm{loc}}$ are differentiable at any $t\ge 0$, with derivatives, respectively, $\chi_\epsilon'\circ v_t \partial_t v_t$ and $2\chi_\epsilon \circ v_t \chi_\epsilon'\circ v_t \partial_t v_t$. Now we compute
		\begin{equation}\notag
			\frac{1}{2}\int_\XX \varphi^2 e^{g_t}\partial_t (\chi_\epsilon\circ v_t)^2\dd\mass=\int_\XX \varphi^2 e^{g_t}\chi_\epsilon\circ v_t\chi_\epsilon'\circ v_t\partial_t v_t\dd\mass
		\end{equation}
		and we deduce by  \eqref{veacfd1} that 
		\begin{equation}\label{vfaedcs}
			\sup_{t\in [0,T_1], \ \epsilon\in (0,1)}\bigg| 	\frac{1}{2}\int_\XX \varphi^2 e^{g_t}\partial_t (\chi_\epsilon\circ v_t)^2\dd\mass\bigg|<\infty
		\end{equation}
		and also that 
		\begin{equation}\notag
			\lim_{\epsilon\searrow 0} 	\frac{1}{2}\int_\XX \varphi^2 e^{g_t}\partial_t (\chi_\epsilon\circ v_t)^2\dd\mass=\int_\XX \varphi^2  e^{g_t} v_t\chi_{\{v_t\ge 0\}} \chi_{\{v_t>0\}}\partial_t v_t\dd\mass=\int_\XX \varphi^2 e^{g_t}v_t^+\partial_t v_t\dd\mass.
		\end{equation}
		By what we have observed above, starting from \eqref{vfadvefd},
		\begin{equation}\notag
			0\le 2\int_\XX e^{g_t} (v_t^+)^2|\nabla \varphi|^2\dd\mass+\frac{1}{2} \int_\XX\varphi^2 e^{g_t}  (v_t^+)^2|\nabla g_t|^2\dd\mass -\lim_{\epsilon\searrow 0}\frac{1}{2}\int_\XX \varphi^2 e^{g_t}\partial_t (\chi_\epsilon\circ v_t)^2\dd\mass.
		\end{equation}
		Using \eqref{veacfd1} and \eqref{vfaedcs}, we can integrate the above on $(T_2,T_1)$ and use dominated convergence to deduce that
		\begin{equation}\label{vefadvac}
\begin{split}
				0&\le 2\int_{T_2}^{T_1}\int_\XX e^{g_t} (v_t^+)^2|\nabla \varphi|^2\dd\mass\dd t+\frac{1}{2}\int_{T_2}^{T_1} \int_\XX\varphi^2 e^{g_t}  (v_t^+)^2|\nabla g_t|^2\dd\mass\dd t \\&\quad\quad-\lim_{\epsilon\searrow 0}\frac{1}{2}\int_{T_2}^{T_1}\int_\XX \varphi^2 e^{g_t}\partial_t (\chi_\epsilon\circ v_t)^2\dd\mass\dd t.
				\end{split}
		\end{equation}
		
		We manipulate further the last integral.
		By an integration by parts, exploiting the fact that for every $t\ge 0$ and bounded set $B$, $t\mapsto e^{g_t}\chi_B\in L^\infty$ is differentiable,
		\begin{align*}
			\int_{T_2}^{T_1}\int_\XX \varphi^2 e^{g_t}\partial_t (\chi_\epsilon\circ v_t)^2\dd\mass\dd t&=\int_{T_2}^{T_1} \partial_t \int_\XX \varphi^2 e^{g_t} (\chi_\epsilon\circ v_t)^2\dd\mass\dd t-\int_{T_2}^{T_1} \int_\XX\varphi^2 \partial_t e^{g_t} (\chi_\epsilon\circ v_t)^2\dd\mass\dd t
			\\&= \int_\XX \varphi^2 e^{g_{T_1}}(\chi_\epsilon \circ v_{T_1})^2\dd\mass- \int_\XX \varphi^2 e^{g_{T_2}}(\chi_\epsilon \circ v_{T_2})^2\dd\mass \\&\quad\quad-\int_{T_2}^{T_1} \int_\XX \varphi^2 \partial_t e^{g_t} (\chi_\epsilon\circ v_t)^2\dd\mass\dd t.
		\end{align*}
		Assume for the moment that $v_{T_2}\le 0\ \mass$-a.e. 
		If we let $\epsilon\searrow 0$, keeping in mind \eqref{veacfd1} and the assumption that we have just made,
		\begin{equation}\notag
			\lim_{\epsilon_\searrow 0} \int_{T_2}^{T_1}\int_\XX \varphi^2 e^{g_t}\partial_t (\chi_\epsilon\circ v_t)^2\dd\mass\dd t=\int_\XX \varphi^2 e^{g_{T_1}} (v_{T_1}^+)^2\dd\mass-\int_{T_2}^{T_1} \int_\XX\varphi^2 \partial_t e^{g_t} (v_t^+)^2\dd\mass \dd t.
		\end{equation}
		Inserting this in \eqref{vefadvac},
		\begin{align*}
			\frac{1}{2}\int_\XX \varphi^2 e^{g_{T_1}} (v_{T_1}^+)^2\dd\mass&
			\le 2\int_{T_2}^{T_1}\int_\XX e^{g_t} (v_t^+)^2|\nabla \varphi|^2\dd\mass\dd t\\
			&\quad	
			+\frac{1}{2}\int_{T_2}^{T_1} \int_\XX\varphi^2 e^{g_t}  (v_t^+)^2|\nabla g_t|^2\dd\mass\dd t +\frac{1}{2}\int_{T_2}^{T_1} \int_\XX\varphi^2 \partial_t e^{g_t} (v_t^+)^2\dd\mass\dd t\\
			&=  2\int_{T_2}^{T_1}\int_\XX e^{g_t} (v_t^+)^2|\nabla \varphi|^2\dd\mass\dd t,
		\end{align*}
		where we used \eqref{bvfeacds} for the last equality.
		
		Now we choose, for $R>1$, $\varphi(x)=\varphi_R(x)=(1\wedge (R+1-\dist(x,\bar x)))\vee 0$. Hence, by the above,
		\begin{equation}\notag
			\int_{B_{R}(\bar x)} e^{g_{T_1}}(v_{T_1}^+)^2\dd\mass\le 4\int_{T_2}^{T_1} \int_{B_{R+1}(\bar x)\setminus B_{R}(\bar x)} e^{-\frac{\dist^2(x,\bar x)}{8{(T_1-T_2)}}} (v_t)^2\dd\mass\dd t.
		\end{equation}
		Now, thanks to \eqref{brrvfedac} and \eqref{cveadfsc},  we see that, if $T_1-T_2\le \Delta_0$, where $\Delta_0$ is small enough (not depending on $T_1$ and $T_2$), we have that $e^{-\frac{\dist^2(x,\bar x)}{8(T_1-T_2)}} (v_t)^2$ is uniformly bounded in $L^1$, so that we can let $R\nearrow \infty$ in  the above equation, together with dominated convergence, to deduce that $\int_\XX e^{g_{T_1}} (v_{T_1}^+)^2\dd\mass=0$, that is, $v_{T_1}\le 0\ \mass$-a.e. 
		
		Recall that in the paragraph we have assumed that $v_{T_2}\le 0\ \mass$-a.e. We can then plug in $T_2=0$, as $v_0=-u \bigg(N+4\log\Big(\frac{\|u\|_{L^\infty}}{u}\Big)\bigg)\le 0\ \mass$-a.e.\ and obtain that  $v_{\Delta_0}\le 0\ \mass$-a.e. We can then repeat the argument, starting at $T_2=\Delta_0$ and then finally iterate,  to reach the conclusion.
	\end{proof}

	\subsection{Proof of Theorem \ref{grfeac}}
		Item $\it (1)$ is  \cite[Theorem 1.2]{JiangZhang16}. We improved the statement from ``for $\mass\otimes\mathcal{L}^1$-a.e.\ $(x,t)$'' to ``for every $t>0$, for $\mass$-a.e.\ $x$'' thanks to an immediate continuity argument.
		Item $\it (2)$ is  \cite[Theorem 1.2]{Jiang15}, or  \cite{GarofaloMondino14} in the  case of finite mass.
		
		We prove now item $\it (3)$. We claim that the conclusion follows from Lemma \ref{hamlem}    (and the fact that $s\le e^s-1$). Indeed, take $(u_k)_k\subseteq L^1\cap L^\infty$ such that $u_k\rightarrow u$ in $L^q$, with $\|u_k\|_{L^\infty}\le \|u\|_{L^\infty}$. Hence, by Lemma \ref{hamlem}, the claim holds for $u_k$, for any $t>0$.  Now fix $t>0$ and notice that $\Delta h_t u_k\rightarrow \Delta h_t u$ in $L^q$, as a consequence of  the Gaussian bound on the heat kernel (see e.g.\  \cite[Theorem 3.3]{Jiang15}). Moreover,   $(h_t u_k)_k$ locally uniformly converges to the (strictly) positive function $h_t u$. Then, \eqref{veafdcs}  holds also for $u$ by approximation.\qed

	\subsection{Proof of Theorem \ref{hkest}}
	Before proving Theorem \ref{hkest}, we need a technical proposition which describes the  derivatives of the function $f_t$.

	\begin{prop}\label{prop28}
		Let  $(\XX,\dist,\mass)$ be an $\RCD(K,N)$ space. Then, for every $x\in\XX$ and $t>0$, $f_t(x,\,\cdot\,)\in D_{\mathrm{loc}}(\Delta)$, and the following equalities hold, 
		\begin{align*}
			&\nabla_y f_t(x,t)=-\frac{\nabla_y \rho_t(x,y)}{\rho_t(x,y)},\\
			&\Delta_y f_t(x,y)=-\frac{\Delta_y\rho_t(x,y)}{\rho_t(x,y)}+\frac{|\nabla_y \rho_t|^2(x,y)}{\rho_t^2(x,y)}=-\frac{\Delta_y\rho_t(x,y)}{\rho_t(x,y)}+|\nabla_y f_t|^2(x,y),
		\end{align*}
		for $\mass$-a.e.\ $y\in\XX$. 

 Moreover, for every $x\in\XX$ and $t>0$, the curve $t\mapsto f_t(x,\,\cdot\,)\in L^2_{\mathrm{loc}}(\XX)$ is locally absolutely continuous in $(0,\infty)$ and is differentiable at every $t>0$,  in the sense that for every bounded set $B$, $t\mapsto f_t(x,\,\cdot\,)\chi_B$ has the properties just stated. Denoting with the dot the derivative in time with respect to the $y$ variable, we have that for every $x\in\XX$ and $t>0$, $\dot f_t(x,y)\in  D_{\mathrm{loc}}(\Delta)$ with
		\begin{align*}
			\dot{f_t}(x,y)&=-\frac{\dot{\rho}_t(x,y)}{\rho_t(x,y)}-\frac{N}{2t}=-\frac{\Delta_y \rho_t(x,y)}{\rho_t(x,y)}- \frac{N}{2t}=\Delta _yf_t(x,y)-|\nabla_y f_t|^2(x,y)-\frac{N}{2t},\\
			\nabla_y \dot f_t(x,y)&=-\frac{\nabla_y\Delta_y\rho_t(x,y)}{\rho_t(x,y)}+\frac{\Delta_y\rho_t(x,y)\nabla_y\rho_t(x,y)}{\rho_t^2(x,y)},\\
			\Delta_y \dot f_t {(x,y)}&=-\frac{\Delta_y\Delta_y\rho_t{(x,y)}}{\rho_t{(x,y)}}+\frac{(\Delta_y\rho_t{(x,y)})^2}{\rho_t^2{(x,y)}}+2\frac{\nabla_y\Delta_y\rho_t{(x,y)}\,\cdot\,\nabla_y \rho_t{(x,y)}}{\rho_t^2{(x,y)}}\\&\quad\quad-2\frac{\Delta_y\rho_t{(x,y)}|\nabla_y\rho_t|^2{(x,y)}}{\rho_t^3{(x,y)}},
		\end{align*}
		for $\mass$-a.e.\ $y\in\XX$. 
		
		Finally, for every $x\in\XX$ and $t>0$, the curve $t\mapsto \nabla_y f_t(x,\,\cdot\,)\in L^2_{\mathrm{loc}}(T\XX)$ is locally absolutely continuous in $(0,\infty)$ and is differentiable at every $t>0$, with 
		\begin{equation}\notag
			\partial_t \nabla_y f_t(x,y)=\nabla_y\dot f_t(x,y)\quad\text{for $\mass$-a.e.\ $y\in\XX$},
		\end{equation}
		and the curve $t\mapsto \Delta_y f_t(x,\,\cdot\,)\in L^2_{\mathrm{loc}}(\XX)$ is locally absolutely continuous in $(0,\infty)$ and is differentiable at every $t>0$, with 
\begin{equation}\notag
	\partial_t \Delta_y f_t(x,y)=\Delta_y\dot f_t(x,y)\quad\text{for $\mass$-a.e.\ $y\in\XX$}.
\end{equation}
	\end{prop}
	\begin{proof}
		The proof follows from the well-known properties of heat kernel, taking into account the  bounds for $\rho_t(x,y)$ and its derivatives. We prove only the claim about $\partial_t\nabla_y f_t(x,\,\cdot\,)$. We simplify the notation, fixing $x\in\XX$ and $t>0$, the differential operators are with respect to the $y$ variable. We have,
		\begin{align*}
			\nabla f_{t+h}-\nabla f_t=-\frac{\nabla( \rho_{t+h}- \rho_t)}{\rho_{t+h}}-\nabla\rho_t \Big(\frac{1}{\rho_{t+h}}-\frac{1}{\rho_t}\Big).
		\end{align*}
		Moreover, integrating by parts, for $h\in (0,1)$,
		\begin{align*}
			\||\nabla\rho_{t+h}-\nabla \rho_t|\|_{L^2}^2\le \frac{1}{2}\|\rho_{t+h}-\rho_{t}\|_{L^2}^2+\frac{1}{2}\|\Delta\rho_{t+h}-\Delta\rho_{t}\|_{L^2}^2,
		\end{align*}
	which give the absolute continuity statement, by the local absolute continuity of the curves $(0,\infty)\ni t\mapsto \rho_t\in L^2$ and  $(0,\infty)\ni t\mapsto\Delta \rho_t\in L^2$. By similar considerations,
			\begin{align*}
		\||\nabla\rho_{t+h}-\nabla \rho_t-h\nabla\partial_t \rho_t|\|_{L^2}^2\le \frac{1}{2}\|\rho_{t+h}-\rho_{t}-h\partial_t\rho_t\|_{L^2}^2+\frac{1}{2}\|\Delta\rho_{t+h}-\Delta\rho_{t}-h\Delta\partial_t\rho_t\|_{L^2}^2
	\end{align*}
	 yields the differentiability statement.
	\end{proof}
	
\bigskip

Now we are ready to prove Theorem \ref{hkest}.
		In the proof, differential operator are implicitly to be understood with respect to the $y$ variable.
		We start from some preliminary observations. 
		
		To begin with, we make the dependence on the curvature explicit in the Gaussian bounds. Instead of carefully tracing the constants in the proof of \cite{jiang2014heat}, we use a scaling argument. Consider then $(\XX,\dist,\HH^n)$ and also the rescaled space
		$(\XX,\delta\dist,\delta^n\HH^n_\XX)=(\XX,\delta\dist,\HH^n)$. If $\tilde \rho$ denotes the heat kernel on the rescaled space, we have $\rho_t(x,y)=\delta^n\tilde\rho_{\delta^2 t}(x,y)$, and so we can use the Gaussian bounds on the rescaled space, which is a non-collapsed $\RCD(-(n-1),n)$ space, to deduce that, for a constant $C=C(n,v,\epsilon)$ as in the statement (notice that the dependence on $v$ is not needed for this inequality),
		\begin{equation}\label{sharpgaussian}
			\frac{1}{C\HH^n(B_{\sqrt{t}}(x))}\exp(-\frac{\dist^2(x,y)}{(4-\epsilon)t}-C\delta^2t)\le \rho_t(x,y)\le \frac{C}{\HH^n(B_{\sqrt{t}}(x))}\exp(-\frac{\dist^2(x,y)}{(4+\epsilon)t}+C\delta^2t),
		\end{equation}
		for every $x,y\in\XX$ and $t>0$.

		Now, we want to bound  $\HH^n(B_{\sqrt{t}}(x))$, for $x\in\XX$ and $t\in (0,\sqrt{10}\delta^{-1})$.  We recall that by Bishop--Gromov, for every $0<r<R$,
		\begin{equation}\label{vfedcs}
			\frac{\HH^n(B_R(x))}{\HH^n(B_r(x))}\le \frac{\int_0^R \sinh(\delta s)^{n-1}\dd s}{\int_0^r \sinh(\delta s)^{n-1}\dd s}= \frac{\delta^{-1}\int_0^{\delta R} \sinh( s)^{n-1}\dd s}{\int_0^{ r}\sinh(\delta s)^{n-1}\dd s},
		\end{equation}
		so that 
		\begin{equation}\notag
			\HH^n(B_R(x))\le \frac{\HH^n(B_r(x))}{\omega_nr^n}\frac{\omega_n \delta^{n-1} r^n}{\int_0^r \sinh(\delta s)^{n-1}\dd s} \frac{\int_0^{\delta R} \sinh( s)^{n-1}\dd s}{(\delta R)^n}  R^n.
		\end{equation}
		Now, we can let $r\searrow 0$ to see that the first factor converges to $\theta_n[\XX,\dist,\HH^n](x)\le 1$, the second factor converges to  $n\omega_n$ by l'Hôpital rule, and the third factor is bounded for $R\le \sqrt{10}\delta^{-1}$ by a constant that depends only on $n$. All in all,
		\begin{equation}\label{veacd}
			\HH^n(B_{\sqrt{t}}(x))\le C t^{n/2} \quad\text{for every $x\in\XX$ and $t\in(0, 10\delta^{-2})$}.
		\end{equation}
		For the bound from below, notice that, if $x\in B_{10\delta^{-1}}(p)$, then $B_{\delta^{-1}}(p)\subseteq B_{11\delta^{-1}}(x)$, so that,
		$\HH^n (B_{11\delta^{-1}}(x))\ge v \delta^{-n}$. Again by Bishop--Gromov,  for $r\in (0,\sqrt{10}\delta^{-1})$, (plugging $R=11\delta^{-1}$ in \eqref{vfedcs})
		\begin{equation}\notag
\begin{split}
				\HH^n(B_r(x))&\ge{\HH^n(B_{11\delta^{-1}}(x))} \frac{\int_0^r \sinh(\delta s)^{n-1}\dd s}{\int_0^{11\delta^{-1}}\sinh(\delta s)^{n-1}\dd s}\ge C^{-1}\delta^{-n}\frac{\int_0^{r\delta} \sinh( s)^{n-1}\dd s}{\int_0^{11}\sinh( s)^{n-1}\dd s}\\
				&\ge r^n C^{-1} \frac{\int_0^{r\delta} \sinh( s)^{n-1}\dd s}{(r\delta)^n}
\end{split}
		\end{equation}
		so that, as the last factor is bounded for $r\in (0,10\delta^{-1})$,
		\begin{equation}\label{vfdcsxas}
			\HH^n(B_{\sqrt{t}}(x))\ge C^{-1}t^{n/2}\quad\text{for every $x\in B_{10\delta^{-1}(p)}$ and $t\in(0, 10\delta^{-2})$}.
		\end{equation}
		
		Now we fix $a=a(\epsilon)\in (0,1)$, depending only upon $\epsilon$, so that $1+a\le \frac{4-\epsilon/2}{4-\epsilon}$.
		Also, we use the Gaussian bounds (with $\epsilon/2$ in place of $\epsilon$) and the doubling inequality (as stated in \eqref{vfedcs}) to  bound from above, for $t\in (0,10\delta^{-2})$ and $x,y\in\XX$,
		\begin{equation}\notag
			\frac{\|\rho_{at}(x,\,\cdot\,)\|_{\infty}}{h_{t}\rho_{at}(x,\,\cdot\,)(y)}=	\frac{\|\rho_{at}(x,\,\cdot\,)\|_{\infty}}{\rho_{(1+a)t}(x,y)}\le C\frac{\HH^n(B_{\sqrt{(1+a)t}}(x))}{\HH^n(B_{\sqrt{at}}(x))} e^{\frac{\dist^2(x,y)}{(4-\epsilon/2)(1+a)t}+C\delta^2 t}\le C e^{\frac{\dist^2(x,y)}{(4-\epsilon/2)(1+a)t}}
		\end{equation} 
		so that, for $x,y\in\XX$ and $t\in (0,10\delta^{-2})$,
		\begin{equation}\label{gteafr}
			\log\bigg( \frac{\|\rho_{at}(x,\,\cdot\,)\|_{\infty}}{h_{t}\rho_{at}(x,\,\cdot\,)(y)}\bigg)\le C+\frac{\dist^2(x,y)}{(4-\epsilon/2)(1+a)t}\le C+\frac{\dist^2(x,y)}{(4-\epsilon)(1+a)^2t}
		\end{equation}
		where we also used the choice of $a$.
		
		Now we are ready to start the proof of the theorem. For item $\it (1)$, we simply use the Gaussian upper bound to deduce
		\begin{equation}\notag
			f_t(x,y)\le C+ \log(\HH^n(B_{\sqrt{t}}(x)))+\frac{\dist^2(x,y)}{(4-\epsilon)t}-\frac{n}{2}\log t\le C+\frac{\dist^2(x,y)}{(4-\epsilon)t}
		\end{equation}
		where we also used \eqref{veacd} in the second inequality. For the lower bound, we argue similarly, but using \eqref{vfdcsxas} in  place of \eqref{veacd}.
		
		Now we prove item $\it (2)$.  We start from item $\it (1)$ of Theorem \ref{grfeac}, with $u=\rho_{at} (x,\,\cdot\,)$, so that $h_t u(y)=\rho_{(1+a)t}(x,y)$. We therefore obtain 
		\begin{equation}\notag
			t|\nabla_y \log \rho_{(1+a)t}|^2(x,y)\le (1+2(n-1)\delta^2t)  \bigg(C+\frac{\dist^2(x,y)}{(4-\epsilon)(1+a)^2t}\bigg)\quad\text{for $\mass$-a.e.\ $y$},
		\end{equation}
		where we used \eqref{gteafr}. This reads, for $t\in (0,10\delta^{-2})$
		\begin{equation}\notag
			t|\nabla_y f_{t}|^2(x,y)\le (1+a)(1+2(n-1)\delta^2t)  \bigg(C+\frac{\dist^2(x,y)}{(4-\epsilon)(1+a)t}\bigg)\quad\text{for $\mass$-a.e.\ $y$},
		\end{equation}
		which is the claim of item $\it (2)$.
		
		Now we turn the upper bound of item $\it (3)$, recalling that 
		$\Delta_y f_t(x,y)=-\frac{\Delta_y\rho_t(x,y)}{\rho_t(x,y)}+|\nabla_y f_t|^2(x,y)$ for $\mass$-a.e.\ $y$.
		We start recalling   item $\it (2)$ of Theorem \ref{grfeac} (with $\rho_t(x,\,\cdot\,)$ in place of $u$), that implies, for $\mass$-a.e.\ $y\in\YY$,
		\begin{equation}\notag
			-\frac{\Delta_y \rho_{2t}(x,y)}{\rho_{2t}(x,y)}\le e^{-2 (n-1)\delta^2 t/3}\frac{n(n-1)\delta^2}{3}\frac{e^{4(n-1)\delta^2t/3}}{e^{2(n-1)\delta^2 t/3}-1}\le \frac{C}{t}\quad\text{for $2t\in (0,10\delta^{-2})$},
		\end{equation}
		that, thanks to item $\it (2)$, implies the upper bound.
		
		We conclude by showing the lower bound of item $\it (3)$. By item $\it (3)$ of Theorem \ref{grfeac} and \eqref{gteafr}, for $\rho_{a t}(x,\,\cdot\,)$  in place of $u$ again,
		\begin{equation}\notag
			t\frac{\Delta_y \rho_{(1+a)t}(x,y)}{\rho_{(1+a)t}(x,y)}\le e^{(n-1)\delta^2 t}\Big(  
			C+4\frac{\dist^2(x,y)}{(4-\epsilon)(1+a)^2t}\Big)\quad\text{for $\mass$-a.e.\ $y$},
		\end{equation}
		so that 
		\begin{equation}\notag
			t\Delta_y f_t(x,y)\ge -e^{(n-1)\delta^2 t/(1+a)}\Big(C+4\frac{\dist^2(x,y)}{(4-\epsilon)t}\Big)\quad\text{for $\mass$-a.e.\ $y$},
		\end{equation}
		which concludes the proof.\qed

		\section{Perelman's entropy}
Now we deal with the computations concerning  Perelman's entropy functional. Recall first Definition \ref{vefdca} and Definition \ref{defnentropy}. 

We define a modified version of the entropy by multiplying the integrand by a compactly supported regular function $\varphi$. 
This approach simplifies the computations, as the heat kernel is uniformly bounded from below on the support of $\varphi$.
		\begin{defn}\label{defnentropy2}
In the setting of Definition \ref{defnentropy}, if $\varphi\in \LIP_{\mathrm{bs}}(\XX)$, define also
		\begin{equation}\notag
			\mathcal{W}_{t,\varphi}(x)\defeq \int_\XX \varphi(y)W_t(x,y) \rho_t(x,y)\dd\mass(y).
		\end{equation} 
	\end{defn}
		For the following definition, we restrict ourselves to the non-collapsed case, as we are aware of applications only in this setting (\cite{CJNrect}).
	\begin{defn}\label{defnentropya}
		Let  $(\XX,\dist,\HH^n)$ be a non-collapsed $\RCD(-(n-1)\delta^2,n)$ space, for $\delta>0$, and let $p\in\XX$. Assume that $v\in (0,1)$ is such that $\HH^n(B_r(p))\ge v r^n$ for any $r\in(0,\delta^{-1})$.
		Let $\varphi_\delta$ as in Lemma \ref{goodcutoff} for $p$ and $r=\delta^{-1}$, and define, for $x\in B_{\delta^{-1}/2}(p)$,
		\begin{equation}\notag
\begin{split}
				\mathcal{W}_{t}^\delta(x)&\defeq \mathcal{W}_{t,\varphi_\delta}(x)-(n-1)\delta^2\int_0^t 4s\int_\XX \varphi_\delta(y)|\nabla_y f_s|^2(x,y)\rho_s(x,y)\dd\mass(y)\\
				&\quad\quad-D(n,v)\delta^2 \int_0^t e^{-1/(100 \delta^2 s)}\dd s,
\end{split}
		\end{equation}
		where $D(n,v)$ is a parameter to be chosen, see Corollary \ref{corgrfdsa}.
	\end{defn}
	Notice that  the integrals in Definition \ref{defnentropy}, Definition \ref{defnentropy2} and Definition \ref{defnentropya} are well defined, thanks to  \eqref{vaawecaweak}  and both the  Gaussian bounds for the heat kernel. 
	The key observation here, is that for every $\eta\in (0,1/6)$, we have that, for $t\in (0,T)$,
	\begin{equation}\label{frcaes}
		\int_\XX e^{\eta \frac{\dist^2(x,y)}{t}} \rho_t(x,y)\le C(K,N,T).
	\end{equation}
	In particular, for every $x\in\XX$, we have that 
	\begin{equation}\label{vefvdsac}
		|W_t(x,\,\cdot\,)|\rho_t(x,\,\cdot\,)\in L_{\mathrm{loc},t}^\infty((0,\infty),L^1_y)\quad\text{and}\quad t|\nabla_y f_t|^2(x,\,\cdot\,)\rho_t(x,\,\cdot\,)\in L_{\mathrm{loc},t}^\infty([0,\infty),L^1_y),
	\end{equation}
	where, by $h_t\in L_{\mathrm{loc},t}^\infty([0,\infty),L^1)$, we mean that for any $a>0$,
	$\sup_{t\in (0,a) }\| h_t\|_{L^1}<\infty$.
	Moreover, if the space is non-collapsed (i.e.\ $N$ coincides with the essential dimension of the space), then both the functions in the above equation belong to $L_{\mathrm{loc},t}^\infty([0,\infty),L^1_y)$ by Theorem \ref{hkest}.

\subsection{Derivative of the entropy}\label{bvefdcsacsd}
Theorem \ref{entrointro} contains the  formula for the derivative of the entropy in a particular case. We are going first to prove a more general version, i.e.\ the formula for the derivative of the modified entropy as in Definition \ref{defnentropy2}, in which the presence of the cut-off function   is helpful in order to do the necessary computations. We will get rid of the cut-off function only at the very end, through a limiting argument,  in Corollary \ref{bvrfdca}.

	\begin{thm}\label{vefadcas}
		Let  $(\XX,\dist,\mass)$ be an $\RCD(K,N)$ space of essential dimension $n$. Let $\varphi\in \LIP_{\mathrm{bs}}(\XX)\cap D(\Delta)$, with $\Delta\varphi\in L^\infty$. Then, for every $x\in\XX$, the curve $t\mapsto\mathcal{W}_{t,\varphi}(x)$ is locally absolutely continuous in $(0,\infty)$, and is differentiable at every $t>0$ with 
\begin{equation}\notag
	\begin{split}
	\partial_t \mathcal{W}_{t,\varphi}(x)&=-2t\int\varphi(y)\rho_t(x,y)\Ric_N(\nabla_y f_t,\nabla_y f_t)(x,\dd y) \\
	&\quad\quad-2t\int\varphi(y)\Big| \hess_y f_t(x,y)-\frac{1}{2t}g(y)\Big|^2\rho_t(x,y)\dd\mass(y)\\
	&\quad\quad-2t\frac{1}{N-n}\int\varphi(y)\Big(\big(\tr(\hess_y f_t(x,y))-\Delta_y f_t(x,y)\big)+\frac{N-n}{2t}\Big)^2\rho_t(x,y)\dd\mass(y)\\
&\quad\quad +\int_\XX \Delta\varphi(y)W_t(x,y)\rho_t(x,y)\dd\mass(y).
	\end{split}
\end{equation}
	\end{thm}
	In the above formula, in the case in which the space is non-collapsed (i.e.\ $N=n$), the integral multiplied by $\frac{1}{N-n}$ has to be interpreted as $0$.
	
	If we formally insert $\varphi \equiv1$ in Theorem \ref{vefadcas}, we obtain the statement of the following corollary. 
		\begin{cor}\label{bvrfdca}
		Let  $(\XX,\dist,\mass)$ be an $\RCD(K,N)$ space of essential dimension $n$. Then, for every $x\in\XX$, the curve $t\mapsto\mathcal{W}_{t}(x)$ is locally absolutely continuous in $(0,\infty)$, and is differentiable at every $t>0$ with 
\begin{equation}\notag
	\begin{split}
		\partial_t \mathcal{W}_{t}(x)&=-2t\int\rho_t(x,y)\Ric_N(\nabla_y f_t,\nabla_y f_t)(x,\dd y) \\
		&\quad\quad-2t\int\Big| \hess_y f_t(x,y)-\frac{1}{2t}g(y)\Big|^2\rho_t(x,y)\dd\mass(y)\\
		&\quad\quad-2t\frac{1}{N-n}\int\Big(\big(\tr(\hess_y f_t(x,y))-\Delta_y f_t(x,y)\big)+\frac{N-n}{2t}\Big)^2\rho_t(x,y)\dd\mass(y)\\
		&\le -2Kt \int |\nabla_y f_t|^2\rho_t(x,y)\dd\mass(y).
	\end{split}
\end{equation}
	\end{cor}
Again, in the case in which the space is non-collapsed (i.e.\ $N=n$), the integral multiplied by $\frac{1}{N-n}$ has to be interpreted as $0$. 
	
	We also have the following corollary, cf.\ \cite[Section 4.6]{CJNrect}.
	\begin{cor}\label{corgrfdsa}
		Let  $(\XX,\dist,\HH^n)$ be a non-collapsed $\RCD(-(n-1)\delta^2,n)$ space, for $\delta>0$, and let $p\in\XX$. Assume that $v\in (0,1)$ is such that $\HH^n(B_r(p))\ge v r^n$ for any $r\in(0,\delta^{-1})$. Then, for any $t\in (0,\delta^{-2})$ and $x\in B_{\delta^{-1}/2}(p)$, choosing $D(n,v)$ depending only upon $n$ and $v$, we have that the curve $t\mapsto\mathcal{W}_{t}^\delta(x)$ is locally absolutely continuous in $(0,\delta^{-2})$,  is differentiable at every $t>0$, and is non-increasing. 
	\end{cor}
	\subsection{Proof of the main results}
	The goal of this section is to prove the results of Section \ref{bvefdcsacsd}.
	\begin{proof}[Proof of Theorem \ref{vefadcas}]
		Throughout the proof, we will implicitly use the conclusions of Proposition \ref{prop28}, and we recall the bounds on  $\rho_t(x,y)$ and its derivatives. Then, we easily gain absolute continuity and differentiability of  $	 \mathcal{W}_{t,\varphi}(x)$.
We simplify the notation, fixing $x\in\XX$ and $t>0$, the differential operators are with respect to the $y$ variable and integration is with respect to $\dd\mass(y)$ on the whole space $\XX$, except for the terms involving $\DDelta\frac{|\nabla f|^2}{2}$, in which case the integration is with respect to $\DDelta\frac{|\nabla f|^2}{2}(x,\dd y)$. We follow morally the same computations as in \cite{NiEntropy}, and everything is justified by the presence of the cut-off function $\varphi$ and Proposition \ref{prop28}.

		
		 We explicitly differentiate $W_t$ in time as a function $(0,\infty)\rightarrow L^2_{\mathrm{loc}}$, to obtain
		\begin{align*}\notag
			\dot W_t&=2\Delta f+	2t\Delta \dot f-|\nabla f|^2-2t\nabla f\,\cdot\,\nabla \dot f+\dot f\\
			&=( 2\Delta f -|\nabla f|^2 +\dot f) +2t\Delta \dot f-2t\nabla f\,\cdot\,\nabla \dot f\\
		&=\Big(3\Delta f-2|\nabla f|^2-\frac{N}{2t}\Big)+\Big( 2t\DDelta\Delta f-2t \DDelta |\nabla f|^2\Big)-2\nabla f\,\cdot \nabla W_t+2|\nabla f|^2+2t\nabla f\,\cdot\,\nabla \Delta f.
		\end{align*}
				Notice also that, by Lemma \ref{bgrfdva1} and a locality argument, $W_t\in D_{\mathrm{loc}}(\DDelta)$, with
		\begin{equation}
			\DDelta W_t=2t\DDelta \Delta f-t\DDelta|\nabla f|^2+\Delta f.
		\end{equation}
		Hence,
		\begin{equation}
			(\partial_t-\DDelta) W_t=-2t\DDelta\frac{|\nabla f|^2}{2}+2t\nabla f\,\cdot\,\nabla \Delta f+2\Delta f-2\nabla f\,\cdot\,\nabla W_t-\frac{N}{2t}.
		\end{equation}
		We can then compute
	\begin{align*}
		&\dv{t}\int\varphi W_t\rho_t-\int \Delta\varphi W_t\rho_t=\int \varphi \partial_t( W_t\rho_t)-\int \Delta\varphi W_t\rho_t
=\int \varphi( \partial_t-\DDelta) (W_t\rho_t)
		\\&=\int \varphi(( \partial_t-\DDelta) W_t) \rho_t-2\varphi \nabla W_t\,\cdot\,\nabla \rho_t\\
		&= \int \varphi \Big(-2t\DDelta\frac{|\nabla f|^2}{2}+2t\nabla f\,\cdot\,\nabla \Delta f+2\Delta f-2\nabla f\,\cdot\,\nabla W_t-\frac{N}{2t}\Big)\rho_t+2\int\varphi \nabla W_t\,\cdot\,\nabla f\rho_t\\
		&=\int\varphi \Big(-2t\DDelta\frac{|\nabla f|^2}{2}+2t\nabla f\,\cdot\,\nabla \Delta f+2\Delta f-\frac{N}{2t}\Big)\rho_t\\
		&=-2t\int\varphi \Big(\DDelta\frac{|\nabla f|^2}{2}-\nabla f\,\cdot\,\nabla \Delta f-\frac{\Delta f}{t}+\frac{N}{4t^2}\Big)\rho_t.
	\end{align*}
	Now we conclude by noticing  that by the very definition of $\Ric_N$,
	\begin{equation}\label{cosmetics}
\begin{split}
		&\DDelta\frac{|\nabla f|^2}{2}-\nabla f\,\cdot\,\nabla \Delta f-\frac{\Delta f}{t}+\frac{N}{4t^2}\\
		&\qquad\qquad=\Ric_N(\nabla f,\nabla f)+|\hess f|^2+\frac{(\tr(\hess f)-\Delta f)^2}{N-n}-\frac{\Delta f}{t}+\frac{N}{4t^2}\\
	&\qquad\qquad=\Ric_N(\nabla f,\nabla f)+\Big| \hess f-\frac{1}{2t}g\Big|^2+\frac{\tr(\hess f)}{t}-\frac{n}{4t^2}\\
	&\qquad\qquad\qquad\qquad+\frac{(\tr(\hess f)-\Delta f)^2}{N-n}-\frac{\Delta f}{t}+\frac{N}{4t^2}\\
	&\qquad\qquad=\Ric_N(\nabla f,\nabla f)+\Big| \hess f-\frac{1}{2t}g\Big|^2+\frac{1}{N-n}\Big(\big(\tr(\hess f)-\Delta f\big)+\frac{N-n}{2t}\Big)^2,
\end{split}
	\end{equation}
where $n$ denotes the essential dimension of the space and fractions with denominator $N-n$ have to be understood as $0$ if $N=n$ (which makes sense, as in that case $\tr(\hess f)=\Delta f$).
	\end{proof}

	\begin{proof}[Proof of Corollary \ref{bvrfdca}]
	Fix $\bar x\in\XX$ and take, for $k\ge 1$,  $\varphi_k\in \LIP_{\mathrm {bs}}(\XX)\cap D(\Delta)$ with $\|\nabla\varphi_k\|_{L^\infty}+\|\Delta\varphi_k\|_{L^\infty}\le C(K,N)$ (but independent of $k$) such that $0\le \varphi_k\le 1$ and  $\varphi_k=1$ on $B_k(\bar x)$. The existence of such functions can be proved following the proof of Lemma \ref{goodcutoff} contained in \cite{Mondino-Naber14}, we recall the argument.  Take indeed $\tilde\varphi_k\in\LIP_{\mathrm{bs}}(\XX)$ with $0\le \tilde\varphi_k\le 1$,  $\LIP(\tilde\varphi_k)\le 1$ and $\varphi_k=1$ on $B_{k}(\bar x)$.  As in the proof of \cite{Mondino-Naber14}, see in particular \cite[Equation (3.4) and the one below]{Mondino-Naber14}, $$\|\nabla h_{t_{K,N}}\tilde\varphi_k\|_{L^\infty}+\|\Delta h_{t_{K,N}}\tilde\varphi_k\|_{L^\infty}\le C(t_{K,N}).$$ Hence, we choose $t_{K,N}$ small enough, depending on $K,N$ but not on $k$, and consider $f\circ h_{t_{K,N}}\tilde\varphi_k$, for a suitable function $f\in C^2$. Notice that, in particular, $|\nabla\varphi_k|=\Delta\varphi_k=0\ \mass$-a.e.\ on $B_{k}(\bar x)$.
	
		Take $0<t_1<t_2$. 
		By Theorem \ref{vefadcas},
		\begin{align*}
			\mathcal{W}_{t_2,\varphi_k}-\mathcal{W}_{t_1,\varphi_k}=\int_{t_1}^{t_2} (-2t)	{\bf I}_{k,t}+ \int_{t_2}^{t_1}\int W_t\Delta\varphi_k\rho_t\dd t,
		\end{align*} 
		where we  set, recalling \eqref{cosmetics},
		\begin{align*}\notag
		{\bf I}_{k,t}
		&\defeq\int\varphi_k \bigg(\Ric_N(\nabla f,\nabla f)+\Big| \hess f-\frac{1}{2t}g\Big|^2+\frac{1}{N-n}\Big(\big(\tr(\hess f)-\Delta f\big)+\frac{N-n}{2t}\Big)^2\bigg)\rho_t,
		\\
		&= \int\varphi_k\Big( \DDelta\frac{|\nabla f|^2}{2}-\nabla f\,\cdot\,\nabla \Delta f-\frac{\Delta f}{t}+\frac{N}{4t^2}\Big)\rho_t.
		\end{align*}
		Hence, by \eqref{vefvdsac} and dominated convergence, if we let $k\rightarrow\infty$ in the above,
		\begin{equation}
			\mathcal{W}_{t_2}-\mathcal{W}_{t_1}=\lim_{k\rightarrow\infty} \int_{t_1}^{t_2} (-2t)	{\bf I}_{k,t}\dd t.
		\end{equation}
		Notice that the integrand  in ${\bf I}_{k,t}$ is uniformly  bounded from below by a finite measure, as one can prove that $|\nabla f|^2\rho_t\in L^1$  arguing as for \eqref{vefvdsac} (we are going to prove again this  below Hence, by monotone convergence,
	\begin{equation}\notag
	{\bf I}_{k,t}\rightarrow{\bf I}_{t}\defeq \int \bigg(\Ric_N(\nabla f,\nabla f)+\Big| \hess f-\frac{1}{2t}g\Big|^2+\frac{1}{N-n}\Big(\big(\tr(\hess f)-\Delta f\big)+\frac{N-n}{2t}\Big)^2\bigg)\rho_t.
\end{equation}		


If we prove that $t\mapsto{\bf I}_{t} $  is continuous, we would get the absolute continuity and the differentiability  of $t\mapsto \mathcal{W}_t$. Integrating by parts,
\begin{align*}
	{\bf I}_{k,t}&=\frac{1}{2}\int (\Delta\varphi_k\rho_t +2\nabla\varphi_k\,\cdot\,\nabla\rho_t+\varphi_k\Delta\rho_t)|\nabla f|^2\\
	&\quad\quad+\int \rho_t \nabla\varphi_k\,\cdot\,\nabla f \Delta f+\varphi_k \nabla\rho_t \,\cdot\,\nabla f \Delta f+\varphi_k \rho_t (\Delta f)^2\\
	&\quad\quad +\int \varphi_k  \Big(\frac{N}{4t^2}-\frac{\Delta f}{t}\Big)\rho_t.
\end{align*}
 Fix now $\tau>0$.
Now we bound, for $\mass$-a.e.\ $y\in\XX$, using the Gaussian bounds together with  \eqref{vaawecaweak}, if $t\in (3/4\tau, 4/3\tau)$, and $\dist^2=\dist^2(x,y)$,
\begin{align*}
&( \rho_t+|\nabla\rho_t|+|\Delta\rho_t|)(1+|\nabla f|+|\nabla f|^2+|\nabla f||\Delta f|+|\Delta f|^2+|\Delta f|)\le \frac{C}{\mass(B_{\sqrt{t}}(x))}e^{-\frac{\dist^2}{5t}+\frac{\dist^2}{100t}}\\\
&\quad\quad\le \frac{C}{\mass(B_{\sqrt{t}}(x))} e^{-\frac{\dist^2}{400/57\tau}}=\frac{C}{\mass(B_{\sqrt{t}}(x))} e^{\frac{\dist^2}{400/3\tau}}e^{-\frac{\dist^2}{20/3\tau}}\le C e^{ \frac{\dist^2}{400/3\tau}}\rho_{20/9 \tau}\eqdef G\in L^1,
\end{align*}
where $C=C(K,N,\tau)$, 
and the membership in $G\in L^1$ (notice that $G$ is independent of $t$) follows from \eqref{frcaes}.  Hence, by dominated convergence, we see that 
\begin{equation}\notag
	{\bf I}_{t}=\int \frac{1}{2}\Delta\rho_t |\nabla f|^2+\nabla\rho_t\,\cdot\,\nabla f\Delta f+\rho_t(\Delta f)^2+\frac{N}{4t^2}\rho_t-\frac{\Delta f}{t}\rho_t
\end{equation}
with $|{\bf I}_{t}|\le C\int G$, and that 
\begin{equation}\notag
	|{\bf I}_{t}-{\bf I}_{k,t}|\le C  \int_{\XX\setminus B_k(\bar x)} G\rightarrow 0\quad\text{as $k\rightarrow \infty$, uniformly in $t\in (3/4\tau, 4/3\tau)$}.
\end{equation}
As $t\mapsto {\bf I}_{k,t}$ is continuous continuous for every $k$, the claim follows.	
	\end{proof}

\begin{proof}[Proof of Remark \ref{efdscacsdsa}]Recall from the proof of Corollary \ref{bvrfdca} that $$( \rho_t+|\nabla\rho_t|)(|\nabla f|+|\Delta f|)\in L^1.$$
Therefore, an immediate approximation argument, with the same choice of cut-off functions as in the proof of Corollary \ref{bvrfdca},  yields that 
\begin{equation*}
    \int \Delta f\rho_t=\int |\nabla f|^2\rho_t.\qedhere
\end{equation*}

\end{proof}
	\begin{proof}[Proof of Corollary \ref{corgrfdsa}]
		Recall \eqref{vefvdsac}, so that  
		\begin{equation}\notag
			t\mapsto t \int_\XX\varphi_\delta |\nabla f_t|^2\rho_t
		\end{equation}
		is locally bounded for $t\in [0,\infty)$. Moreover, it is clearly continuous in $t\in (0,\infty)$, by the properties of the heat kernel.
		Hence, local absolute continuity and differentiability follow from Theorem \ref{vefadcas}, together with the expression 
\begin{equation}\label{vefdscads}
\begin{split}
								\partial_t \mathcal{W}_{t}^\delta&=-2t\int_\XX \varphi_\delta\rho_t\dd\Ric(\nabla f_t,\nabla f_t) -2t \int_\XX\varphi_\delta\Big| \hess f_t-\frac{1}{2t}g\Big|^2\rho_t\\
		&\quad\quad+\int_\XX \Delta\varphi_\delta W_t\rho_t-4t(n-1)\delta^2\int_\XX \varphi_\delta|\nabla f_t|^2\rho_t\\
		&\quad\quad-D(n,v)\delta^2 e^{-1/(100 \delta^2 t)}.
\end{split}
\end{equation}
We have to prove that the right hand side of \eqref{vefdscads} is non-positive.
	By the $\RCD(-(n-1)\delta^2,n)$ condition, 
		\begin{equation}\notag
			-2t\int_\XX \varphi_\delta\rho_t\dd\Ric(\nabla f_t,\nabla f_t)-4t(n-1)\delta^2\int_\XX \varphi_\delta|\nabla f_t|^2\rho_t\le 0.
		\end{equation}
		Now, for $x\in B_{\delta^{-1}/2}(p)$ and $t\in (0,\delta^{-2})$, thanks   to \eqref{vaaweca},
		\begin{align*}
			&\Big|	\int_\XX \Delta\varphi_\delta W_t\rho_t\Big|\le C(n)
			\int_{B_{2\delta^{-1}}(p)\setminus B_{\delta^{-1}}(p)} \delta^2 |W_t|\rho_t\\
			&\quad\quad\le C(n,v)\delta^2 e^{-1/(100 \delta^2 t)}\int_{B_{2\delta^{-1}}(p)\setminus B_{\delta^{-1}}(p)} e^{1/(100 \delta^2 t)}e^{ \frac{\dist^2}{10 t}}\rho_t\\
			&\quad\quad\le C(n,v)\delta^2 e^{-1/(100 \delta^2 t)},
		\end{align*}
		we also used \eqref{frcaes} and the observation that for $x\in B_{\delta^{-1}/2}(p)$ and $y\in B_{2\delta^{-1}}(p)\setminus B_{\delta^{-1}}(p) $, $\dist(x,y)\ge \delta^{-1}/2$. To be more precise, the direct application   of \eqref{frcaes} yields a bound which depends also on $\delta$. However, exploiting the estimates of the proof of Theorem \ref{hkest}, in particular, \eqref{sharpgaussian}, \eqref{veacd}, and \eqref{vfdcsxas}, we obtain the conclusion.
		Hence, the proof is concluded choosing $D=D(n,v)$.
	\end{proof}
	
	\subsection{Proof of Theorem \ref{entrointro}}
	The \textbf{monotonicity} statement of Theorem \ref{entrointro} is proved in Corollary \ref{bvrfdca}.
	This section is then devoted to the proof of the \textbf{rigidity} statement of Theorem \ref{entrointro}. We give a simple ad-hoc  proof.

	\bigskip

		 By the usual scaling properties of the heat flow, we can assume that $\partial_t \mathcal{W}_1=0$, to simplify the notation. We set $h_t(x,y)\defeq t f_t(x,y)$.
		As usual we soften the notation as in the proof of Theorem \ref{vefadcas}.
\medskip\\{\textbf{Step 1}}. First, if we plug  the assumption that $\partial_t\mathcal W_1=0$ into the conclusion of  Corollary \ref{bvrfdca}, 
\begin{equation}\notag
	\begin{split}
	0&=\int\rho_1\dd\Ric_N(\nabla f_1,\nabla f_1)+\int\Big| \hess f_1-\frac{1}{2}g\Big|^2\rho_1\dd\mass(y)\\
		&\quad\quad+\frac{1}{N-n}\int\Big(\big(\tr(\hess f_1)-\Delta f_1\big)+\frac{N-n}{2}\Big)^2\rho_1.
	\end{split}
\end{equation}
Therefore,
\begin{equation}\label{vfdsacas}
	\hess h_1=\frac{g}{2}\quad\mass\text{-a.e.}
\end{equation}
and
\begin{equation}\notag
	\tr(\hess h_1)-\Delta h_1+\frac{N-n}{2}=0\quad\mass\text{-a.e.}
\end{equation}
so that 
\begin{equation}\label{cdsacas}
	{\Delta h_1}=\frac{{N}}{2}\quad\text{$\mass$-a.e.}
\end{equation}

As a side remark, we notice that \eqref{cdsacas} is exactly
\begin{equation}\notag
	\Delta \log(\rho_1)=-\frac{N}{2},
\end{equation}
i.e.\ equality in the Li--Yau inequality at time $t=1$. This is enough to conclude, i.e.\  the Li--Yau inequality is rigid  on non-smooth spaces, but we will not exploit this fact. We will use a different method, tailored to this setting and self-contained.

 Notice also that 
		\begin{equation}\notag
	\nabla |\nabla h_1|^2=\nabla h_1\quad\text{$\mass$-a.e.}
	\end{equation}
	which is a consequence of  \eqref{vfdsacas} and the calculus rules.
	\medskip\\{\textbf{Step 2}}.
	Take $\varphi\in \LIP_{\rm bs}(\XX)\cap D(\Delta)$. We want to show that the function (that depends on $\varphi$, but we will not keep this dependence explicit)
		\begin{equation}\notag
		(0,\infty)\ni	t\mapsto \Phi_t\defeq\int \Delta\varphi  h_t
	\end{equation}
	is constant in $t$. 	First, recall that 
	\begin{equation}\notag
		h_t=tf_t=-t\log(\rho_t)-t\frac{N}{2}\log(t)-t\frac{N}{2}\log(4\pi),
	\end{equation}
	so that, recalling Proposition \ref{prop28},
		\begin{equation}\label{vefdcvacdsa}
	\partial_ t h_t= t\partial_t f_t+f_t=-t\frac{\partial_t \rho_t}{\rho_t}-\frac{N}{2}+\frac{h_t}{t}=\Delta h_t+\frac{h_t-{|\nabla h_t|^2}}{t}-\frac{N}{2}.
\end{equation}		
Moreover,
	\begin{equation}\notag
		\Phi_t=\int \Delta\varphi (-t\log(\rho_t)). 
	\end{equation}
By the Gaussian bounds for the heat kernel, for any  compact set $B\subseteq(0,\infty)\times\XX$, there exists $C=C(\XX,B)>0$ such that, for every $k\ge 0$,
\begin{equation}\notag
	\inf_{(t,y)\in B}|\rho_t(x,y)|\ge C^{-1}\quad\text{and}\quad \sup_{(t,y)\in B}|\partial_t^k\rho_t(x,y)|\le C^k.
\end{equation}
By standard properties of analytic functions, these bounds yield the analyticity of $\Phi_t$, locally on $(0,\infty)$. 

Hence, it is enough to show that ${\partial_t^k \Phi_t}_{|t=1}=0$ for every $k\ge 1$. 
This follows from \eqref{vefdcvacdsa} if we prove  by induction that for every $j\ge 0$, (all the derivatives in time are for functions in $L^2_{\mathrm{loc}}$)
\begin{itemize}
	\item [$A_j:$] ${\nabla 	\partial_t^j\Delta h_t}_{|t=1}=0\ \mass$-a.e.
\item [$B_j:$]${\nabla 	\partial_t^j(h_t-|\nabla h_t|^2)}_{|t=1}=0\ \mass$-a.e.
\end{itemize}
Now we notice that $A_0$ and $B_0$ are verified by \textbf{Step 1}. Assume now that we have verified $A_0,A_1,\dots, A_j$ and $B_0,B_1,\dots,B_j$, we want to show $A_{j+1}$ and $B_{j+1}$. Notice first that differentiating in time \eqref{vefdcvacdsa} $j$ times, the combination of $A_j$ and $B_0,B_1,\dots, B_j$ implies 
\begin{itemize}
	\item [$C_{j}:$] ${	\nabla\partial_t^{j+1} h_t}_{|t=1}=0\ \mass${-a.e.}
\end{itemize}
Arguing as for Proposition \ref{prop28}, the regularity properties of the heat flow allow us to show that 
\begin{align}\label{vacdsxa}
	\partial_t \Delta \partial_t^{i-1}h_t&=	\Delta \partial_{t}^{i}h_t\quad\mass\text{-a.e.\ for every $t\in (\tau,\tau^{-1})$ and $i\ge 1$},\\
	\label{vacdsx0a}
	\partial_t\nabla\partial_t^{i-1}  h_t&=\nabla \partial_{t}^{i}h_t\quad\mass\text{-a.e.\ for every $t\in (\tau,\tau^{-1})$ and $i\ge 1$}.
\end{align}
By \eqref{vacdsxa},  $C_{j}$  yields $A_{j+1}$. 
For what concerns $B_{j+1}$, taking in account $C_{j}$, it is enough to show that ${\nabla 	\partial_t^{j+1}(\nabla h_t\,\cdot\, \nabla h_t)}_{|t=1}=0\ \mass$-a.e. This claim follows from $C_0, C_1,\dots, C_j$ and \eqref{vacdsx0a}, by explicit differentiation.
This concludes the proof of the fact that $\Phi_t$ is constant.
\medskip\\{\textbf{Step 3}}.
Recall that by \textbf{Step 2}, $t\mapsto\Phi_t$ is constant for every $\varphi\in \LIP_{\rm bs}(\XX)\cap D(\Delta)$.  Considering $t\Delta f_t=\Delta h_t$ at time $t=1$ (by the computation in \textbf{Step 1}), this implies that
\begin{equation}\notag
	-\Delta h_t=t\Delta\log \rho_t=-\frac{N}{2}\quad\mass\text{-a.e.\ for every $t>0$}.
\end{equation}
Also,   for every $\varphi\in \LIP_{\rm bs}(\XX)\cap D(\Delta)$,  $\partial_t\Phi_t=0$ for every $t>0$. This means that 
\begin{equation}
	\int \Delta \varphi \partial_t h_t=0\quad\text{for every }t>0.
\end{equation}
Recalling \eqref{vefdcvacdsa}, this implies that for every $t>0$, $\Delta |\nabla h_t|^2=\Delta h_t=\frac{N}{2}\ \mass$-a.e.\ for every $t$. By Corollary  \ref{bvrfdca}, we have the constancy of $\mathcal{W}$.

	Take now  $\varphi\in\LIP_{\mathrm{bs}}(\XX)\cap D(\Delta)$. Integrating by parts,
	\begin{equation}\notag
		\int t\log\rho_t  \Delta\varphi=-\frac{N}{2}\int \varphi\quad\text{for every }t>0.
	\end{equation}
	Now, we can conclude as in {\cite{NiEntropy,KuwadaLi}. We use the Gaussian bounds on $\rho_t$ together with the doubling inequality to see that $t\log(\rho_t)\rightarrow - \dist^2(x,\,\cdot\,)/4$ locally uniformly, as $t\searrow 0$, which is a by now classical Varadhan type short time asymptotic (and follows from the Gaussian bounds for  the heat kernel with the doubling inequality). This implies that  
	\begin{equation}\notag
		\int\frac{ \dist^2(x,\,\cdot\,)}{4}\Delta\varphi=\frac{N}{2}\int \varphi.
	\end{equation}
	Being $\varphi$ arbitrary, this means that $\dist^2(x,\,\cdot\,)\in D_{\mathrm{loc}}(\Delta)$, with $\Delta\dist^2(x,\,\cdot\,) =2N\ \mass$-a.e. Now we can compute, for a.e.\ $r>0$,
	\begin{align*}
		2N\mass(B_r(x))&=\int_{B_r(x)}\Delta\dist^2(x,\,\cdot\,)=\int_{\partial B_r(x)}2|\nabla \dist(x,\,\cdot\,)|\dist(x,\,\cdot\,)\dd|\DIFF\chi_{B_r(x)}|=2r|\DIFF\chi_{B_r(x)}|\\&=2r \partial_r \mass(B_r(x)).
	\end{align*} 
	In this derivation we used the theory of calculus for functions of bounded variation on $\RCD$ spaces, see e.g.\ \cite[
	Theorem 2.4]{bru2019rectifiability} and \cite{BGBV}, in particular \cite[Lemma 4.27]{BGBV}, and the fact that $|\DIFF\chi_{B_r(x)}|(\XX)=\partial_r \mass(B_r(x))$ for a.e.\ $r>0$ follows from the coarea formula applied to the distance function. 
	Hence, for a.e.\ $r>0$,
	\begin{equation}\notag
		\partial_r \frac{\mass(B_r(x))}{r^N}=\frac{\partial_r \mass(B_r(x))}{r^N}-N\frac{\mass(B_r(x))}{r^{N+1}}=0,
	\end{equation}
	which implies that $(0,\infty)\ni r\mapsto \frac{\mass(B_r(x))}{r^N}$ is constant (in particular, the space is not compact). This allows us to conclude the proof by an application of the main result of \cite{DPG16}.\qed
%
%
%


\end{document}